\documentclass[11pt]{amsart}

\usepackage{amsmath, amssymb}
\usepackage{mathtools}
\usepackage[scr=dutchcal,calscaled=1.05,bb=boondox,bbscaled=1.1]{mathalfa}
\usepackage{tikz-cd}
\usepackage{enumerate}
\usepackage{tensor}
\usepackage{color}
\usepackage{graphicx}
\usepackage[top=1.1in, bottom=1in, left=1in, right=1in]{geometry}
\usepackage{hyperref}

\numberwithin{equation}{section}
\newtheorem{thm}{Theorem}[section]
\newtheorem{lemma}[thm]{Lemma}
\newtheorem{prop}[thm]{Proposition}
\newtheorem{cor}[thm]{Corollary}
\newtheorem{conj}[thm]{Conjecture}

\theoremstyle{definition}
\newtheorem{defn}[thm]{Definition}
\newtheorem{question}[thm]{Question}

\theoremstyle{remark}
\newtheorem{rmk}[thm]{Remark}
\newtheorem{const}[thm]{Construction}
\newtheorem{eg}[thm]{Example}

\newcommand{\stoppedarrow}{%
	\mathrel{%
		\text{%
			\ooalign{$\mapsto$\cr\reflectbox{$\mapsto$}\cr}%
		}%
	}%
}
\newcommand{\C}{{\mathbb C}}
\newcommand{\Z}{{\mathbb Z}}
\newcommand{\R}{{\mathbb R}}
\renewcommand{\H}{{\mathbb H}}

\newcommand{\glu}[2]{\tensor[_{#1}]{\cup}{_{#2}}}
\newcommand{\sglu}[2]{\tensor[_{#1}]{\overset{\stoppedarrow}{\cup}}{_{#2}}}
\newcommand{\dglu}[2]{\tensor[_{#1}]{\overset{\rightarrow}{\cup}}{_{#2}}}
\newcommand{\into}{\hookrightarrow}
\newcommand{\F}{\mathcal F}
\newcommand{\W}{\mathcal W}
\newcommand{\B}{\mathcal B}
\renewcommand{\O}{{\mathcal O}}
\newcommand{\K}{{\mathbf K}}
\DeclareMathOperator{\hocolim}{hocolim}
\DeclareMathOperator{\Perf}{Perf}
\DeclareMathOperator{\tw}{tw}
\DeclareMathOperator{\cone}{cone}
\renewcommand{\Re}{{\operatorname{Re}}}

\newcommand{\Ob}{{\mathcal O}{\mathscr b}}
\newcommand{\Fun}{{\mathcal F}\mathscr{un}}

\title{Orlov and Viterbo functors in partially wrapped Fukaya categories}
\author{Zachary Sylvan}

\begin{document}

\begin{abstract}
	We study two functors between (partially) wrapped Fukaya categories. The first is the Orlov functor from the Fukaya category of a stop to the Fukaya category of the ambient sector. We give a geometric criterion for when this functor is spherical in the sense of Anno--Logvinenko. This criterion is a generalization of the situation where the stop comes from a Landau--Ginzburg model.
	
	The second functor is the Viterbo transfer map from a Liouville domain to a subdomain. We show that when the domain and subdomain are independently Weinstein, this functor is a homological epimorphism, which means that it becomes a localization after passing to module categories. This should be compared with a result of Ganatra--Pardon--Shende, which states that the Viterbo map is a genuine localization when the cobordism is Weinstein.
\end{abstract}

\maketitle

\section{Introduction}\label{ch:intro}

% Describe construction of Viterbo map. Describe proof that it's a localization. Explain consequence for regular Lagrangian conjecture: an exact Lagrangian is "algebraically regular" iff the kernel of $V$ is split-generated by an embedded Lagrangian disk. In this case, $W(T^*L)\cong W(X\setminus D)$.

The goal of this paper is to discuss a circle of ideas surrounding two functors between partially wrapped Fukaya categories.

The first of these functors is the \emph{Orlov functor}, which associates to a Lagrangian in a stop the product of that Lagrangian with a transverse arc. When $\sigma\subset\partial_\infty M$ is a Weinstein hypersurface\footnote{For the purposes of this introduction, we take the word \emph{stop} to mean a Liouville hypersurface of a contact manifold.}, this is just the functor which sends a co-core in $\sigma$ to the small disk in $M$ linking the corresponding core. The Orlov functor for a stop generalizes the better-known cup functor on a symplectic Landau--Ginzburg model, which sends a Lagrangian in the fiber to its parallel transport around a large arc.

The second functor is the \emph{Viterbo transfer map}, which is defined for a Liouville subdomain $M^\mathrm{in}\subset M$. Morally speaking, it sends a Lagrangian $L\subset M$ to its intersection $L\cap M^\mathrm{in}$. This was partially defined in \cite{AbouzaidSeidel2010}, but for technical reasons it is difficult to extend that definition to the generality that one would like. Instead, we will use the definition via Orlov functors which appears in \cite{GanatraPardonShende2018} and verify that the definitions are  essentially the same (see Section \ref{sec:abouzaid--seidel}). This latter definition requires a technical assumption of its own, that $M^\mathrm{in}$ ``satisfy stop removal'', but this is fairly mild by current standards: in that same paper the authors prove that all Weinstein domains satisfy stop removal.

\subsection{The Orlov functor}
For a symplectic Landau--Ginzburg model, it is a theorem of Abouzaid and Ganatra \cite{AbouzaidGanatra20XX} that the Orlov functor is spherical in the sense of \cite{AnnoLogvinenko2017}. In this setting, its left (resp. right) adjoint $\cap^L$ (resp. $\cap^R$) is given by wrapping positively (resp. negatively) and intersecting with a fixed fiber, which plays the role of the stop. Moreover, the dual spherical twist and dual spherical cotwist are given up to shift by the symplectomorphisms ``wrap-once'' and ``monodromy'', respectively. By contrast, the Orlov functor for a general stop has neither left nor right adjoints.

We would like to account for this difference by seeing exactly what geometric properties of $W$ are responsible for the above structure. As it turns out, the geometric property that underlies sphericality is the same as the structure which provides the monodromy and wrap-once autoequivalences.

\begin{defn}
	A stop $\sigma\subset\partial_\infty M$ is \emph{swappable} if there is some isotopy of stops $\phi$ from the positive Reeb pushoff $\sigma^+$ to the negative Reeb pushoff $\sigma^-$ which avoids the original stop $\sigma$.
\end{defn}

The fiber at infinity of a Landau--Ginzburg model is the prototypical example of a swappable stop. Indeed, rotation around the $S^1$ in the base precisely gives a long path from $\sigma^+$ to $\sigma^-$. This path has the additional feature of being globally embedded and globally positive, but that turns out not to be necessary.

The definition of swappability makes sense for Lagrangians as well as for stops, where it turns out to have been studied before:

\begin{thm}[\cite{FrauenfelderLabrousseSchlenk2015, Dahinden2017}]\label{thm:contact Bott--Samelson}
	Let $\Lambda=S_p^*Q\subset S^*Q=\partial_\infty T^*Q$ be a swappable cosphere fiber. Then $H^*(Q,\Z)$ is generated as a ring by one element, and $\lvert\pi_1(Q)\rvert\le2$.
\end{thm}

See Lemma \ref{lem:contact swap characterizations} for the relation between our notion of swappability and the notion which appears in the proof of Theorem \ref{thm:contact Bott--Samelson}.

Returning to our objective, let $\sigma$ be a swappable stop of $M$. By following the Reeb flow $R$ from $\sigma$ to $\sigma^+$, the isotopy $\phi$ from $\sigma^+$ to $\sigma^-$, and $R$ again from $\sigma^-$ to $\sigma$, we obtain a self-isotopy of $\sigma$, and we define the $\mathbf W_\phi$ to be the autoequivalence of the partially wrapped Fukaya category $\W(M\setminus\sigma)$ obtained by following this isotopy. The same isotopy also induces a monodromy symplectomorphism on $\sigma$, and we define $\mathbf M_\phi$ to be the resulting autoequivalence of the fully wrapped Fukaya category $\W(\sigma)$.

\begin{thm}\label{thm:swappable stops give spherical Orlov functors}
	Let $\sigma$ be a swappable stop satisfying stop removal. Then the corresponding Orlov functor is spherical, the monodromy autoequivalence $\mathbf M_\phi$ agrees with the shifted dual cotwist $\mathbf M'[-2]$, and the wrap-once autoequivalence $\mathbf W_\phi$ agrees with the dual twist $\mathbf W'$.
	
	In particular, $\mathbf M_\phi$ and $\mathbf W_\phi$ are independent of $\phi$.
\end{thm}

\begin{eg}\label{eg:punctured A^2}
One way of producing new swappable stops from old ones is to pass to a monodromy-invariant subdomain. As an example, consider the Landau--Ginsburg model $((\C^*)^2,W=x+y+1)$ mirror to the toric variety $\C^2$. The fiber $F$ is a pair of pants, where two of the legs are mirror to the coordinate axes and the third leg is mirror to the origin. We can pass to the monodromy-invariant subdomain $F_0$ which is the union of the first two cylinders. Under mirror symmetry, this corresponds to puncturing the origin. We are left with a stop of two components, each of which is independently swappable.

This situation is studied in the setting of FLTZ mirror symmetry \cite{FangLiuTreumannZaslow2011} in \cite{KatzarkovKerr2017}.
\end{eg}

The fact that $\mathbf M_\phi$ and $\mathbf W_\phi$ are independent of $\phi$ has immediate geometric consequence:

\begin{cor}
	If $\psi$ is any symplectomorphism of $M$ fixing $\sigma$ setwise, then the induced autoequivalence $\Psi^M$ of $\W(M,\sigma)$ commutes with $\mathbf W_\phi$, and the induced autoequivalence $\Psi^\sigma$ of $\W(\sigma)$ commutes with $\mathbf M_\phi$.
\end{cor}
\begin{proof}
	The conjugate $\psi^{-1}\phi\psi$ is also a swap of $\sigma$, so
	\[
	(\Psi^M)^{-1}\mathbf W_\phi\Psi^M\cong\mathbf W_\phi,
	\]
	and similarly with $\mathbf M_\phi$.
\end{proof}

We end with a couple observations and speculations about this story.

First, there is an algebraic analog of swappability, that of a \emph{spherical swap}, which appeared without a name in \cite{HalpernLeistnerShipman2016} as the square root of a window shift. It turns out that existence of a spherical swap is equivalent to sphericality (cf. Proposition \ref{prop:spherical swap iff spherical}). The closeness of the algebraic and symplectic definitions thus suggests that spherical Orlov functors which do not come from swappable stops should be fairly exotic.

Going backwards, swaps provide a simple symplectic interpretation of Segal's theorem that all autoequivalences are spherical twists \cite{Segal2017}. Specifically, suppose some autoequivalence $\Phi$ of $\W(\bar M)$ comes from a symplectomorphism $\phi$ of $\bar M$. Then we can enhance the mapping torus
\[
\left(\bar M\times T^*\R\right)/\sim
\]
of $\phi$ with a swappable stop whose monodromy is $\phi$ (note that the dual cotwist of a spherical functor is the twist of its left adjoint).

Finally, it is known that for Lefschetz fibrations, the spherical adjunction $\cap^L\dashv\imath$ is compatible with a relative proper Calabi-Yau structure \cite{Seidel2012,KatzarkovPanditSpaide2017} on the Orlov functor $\imath$. In general, the partially wrapped Fukaya category is not proper, but for a Weinstein pair (i.e. a Weinstein manifold $M$ with Weinstein stop $\sigma$) it is expected to be smooth. In this case, we expect the following conjecture to hold, which I learned from Ganatra.

\begin{conj}
	Let $(M,\sigma)$ be a Weinstein pair such that $M$ and $\sigma$ are individually Calabi-Yau. Then the Orlov functor $\imath_\sigma$ has a natural relative smooth Calabi-Yau structure in the sense of \cite{BravDyckerhoff2019}.
\end{conj}

The assumption that $(M,\sigma)$ is a Weinstein pair, meaning in particular that there are no additional stops lurking in the background, is crucial. Indeed, Brav and Dyckerhoff show in \cite{BravDyckerhoff2019} that when a spherical functor has a relative smooth Calabi--Yau structure, the spherical twist depends only on the target category, but one can modify Example \ref{eg:punctured A^2} to make the two stops have different twists.

\subsection{The Viterbo map}
Consider first the case of a Weinstein subdomain $M^\mathrm{in}\subset\bar M$, meaning that $M^\mathrm{in}$ and the cobordism $\bar M\setminus M^\mathrm{in}$ are both Weinstein. It follows from the cosheaf property that the Viterbo transfer map
\[
\mathcal V\colon\W(\bar M)\to\Perf\,\W(\bar M^\mathrm{in})
\]
is a localization, and that $\ker\mathcal V$ is (split-) generated by the co-cores of the cobordism \cite{GanatraPardonShende2018}. We would like to understand to what extent this remains true if we drop the assumption that $\bar M\setminus M^\mathrm{in}$ is Weinstein.

\begin{question}\label{qn:V a localization?}
	Suppose $M^\mathrm{in}$ and $\bar M$ are individually Weinstein. Must $\mathcal V$ be a localization (up to summands)?
\end{question}

I do not know how to answer this question, but I expect the answer is no. Instead, we will prove a partial result.

\begin{thm}\label{thm:Viterbo is epic}
	Let $M^\mathrm{in}\subset\bar M$ be a Liouville subdomain, and suppose both $\bar M$ and the Liouville completion $\bar M^\mathrm{in}$ satisfy stop removal. Then the Viterbo transfer map
	\[
	\mathcal V\colon\W(\bar M)\to\Perf\,\W(\bar M^\mathrm{in})
	\]
	is a homological epimorphism.
	
	In particular, the image of $\mathcal V$ split-generates $\W(\bar M^\mathrm{in})$.
\end{thm}

Here, the assertion that $\mathcal V$ is a \emph{homological epimorphism} means that its extension
\[
\mathcal V_!\colon\mathrm{Mod}\text{-}\W(\bar M)\to\mathrm{Mod}\text{-}\W(\bar M^\mathrm{in})
\]
is a localization. Such functors do not enjoy all of the useful properties that genuine localizations have, but the difference is reasonably well understood. Indeed, a homological epimorphism becomes a localization as soon as $\mathcal V_!$ has enough perfect modules in its kernel \cite{Neeman1996}. To prove split-generation it would of course be simpler to transfer the open-closed map and apply the generation criterion \cite{Abouzaid2010}, which with modern technology is a reasonably quick and worthwhile exercise. Nonetheless, this is to my knowledge the first written proof of the fact.

We close by recalling the \emph{regular Lagrangian conjecture} \cite{EliashbergGanatraLazarev2018}, which posits that any exact Lagrangian $L$ in a Weinstein manifold $M$ can be made part of the skeleton. If true, this would supply an affirmative answer to Question \ref{qn:V a localization?} in the case where $M^\mathrm{in}$ is a Weinstein neighborhood of $L$. Given that, it seems reasonable to view Question \ref{qn:V a localization?} for Lagrangians as a weakening of the regular Lagrangian conjecture, which for some class of Lagrangians might be tractably attacked by studying $\ker\mathcal V$. More optimistically, one could ask when $\ker\mathcal V_!$ is generated (under filtered colimits) by a single Lagrangian disk $D$. If this happens, then $\bar M\setminus D$ would be Floer theoretically equivalent to $T^*L$.

\subsection{Outline of the paper}
We begin Section \ref{ch:viterbo} with a review of partially wrapped Floer theory, and proceed to prove new isotopy invariance properties. This allows us in Section \ref{sec:abouzaid--seidel} to establish the equivalence between the Abouzaid--Seidel and sectorial constructions of the Viterbo functor.

Section \ref{ch:subdomains and bimodules} concerns various gluing operations on Liouville sectors and their algebraic interpretation. The resulting formulas are equivalent to the gluing formulas in \cite{GanatraPardonShende2018} but simpler for our purposes. We use these to geometrically reinterpret the conclusion of Theorem \ref{thm:Viterbo is epic}, which we prove by observing that two sectors are isotopic.

The final section \ref{ch:spherical} begins with a discussion of spherical functors and their characterization in terms of spherical swaps. We then prove Theorem \ref{thm:swappable stops give spherical Orlov functors} by verifying that a swap of a stop induces a swap of its Orlov functor.

\subsection*{Acknowledgments}
Thank you to Sasha Efimov for patiently explaining homological epimorphisms, Sheel Ganatra for telling me about spherical functors and for ideas about pushforward functors, and Emmy Murphy for suggesting there should be such a thing as Viterbo transfer to ``substops''. The section on swappable stops was inspired by a talk by Guogang Liu at the Princeton/IAS symplectic geometry seminar. I'd also like to thank Mohammed Abouzaid, Denis Auroux, Oleg Lazarev, Vivek Shende, and Paul Seidel for helpful conversations.

This project started while I was a member at the Institute for Advanced Study and was completed with the support of the Simons Foundation through grant \#385573, the Simons Collaboration on Homological Mirror Symmetry.
 
\section{The sectorial Viterbo map}\label{ch:viterbo}

\subsection{Floer theories on Liouville sectors}
% Review stops/sectors. Assume without proof that my and GPS's version of PWr are homotopically canonically equivalent. Define Orlov functor. Construct Viterbo sector. Compute that inclusion of subdomain is fully faithful. Define Viterbo functor.

We work in the setting of Liouville manifolds and Liouville sectors. A \emph{Liouville manifold} is an exact symplectic manifold $(\bar M,\lambda)$ with complete, finite type, convex Liouville vector field $Z$, meaning that all negative trajectories of $Z$ become trapped in a compact set. Writing $\H=\C_{\Re(z)\le1}$ with the radial Liouville form $\frac12(xdy-ydx)$, a \emph{stop} in $\bar M$ with fiber a Liouville manifold $\bar F$ is a proper embedding
\[
\sigma\colon\bar F\times\H\to\bar M
\]
which intertwines the Liouville forms up to a compactly supported exact 1-form. Its \emph{divisor} $D_\sigma$ is the submanifold $\sigma(\bar F\times\{0\})$. A \emph{Liouville sector} is a Liouville manifold with boundary of the form $\bar M\setminus\sigma(\bar F\times\C_{\Re<0})$, which we will abbreviate and write $\bar M\setminus\sigma$, or just $M$ when the stop is understood. For $M$ a Liouville sector, the contact manifold at infinity $\partial_\infty M$ has convex boundary and admits $\partial_\infty D_\sigma$ as a dividing set. In this case, the positive and negative boundaries $\partial^+\partial_\infty M$ and $\partial^-\partial_\infty M$ refer to $\sigma(\bar F\times i\R_{\ge0})$ and $\sigma(\bar F\times i\R_{\le0})$, respectively.

Ganatra, Pardon, and Shende \cite{GanatraPardonShende2017} show that the space of Liouville sectors is homotopy equivalent to the space of pairs $(\bar M,\sigma)$, and the fiber can be recovered as the symplectic reduction of the boundary. Note that they use a slightly more general definition of Liouville sector and refer to our sectors as Liouville sectors with exact boundary.

More practically, they also show that the space of Liouville manifolds $\bar M$ equipped with a Liouville hypersurface of $\partial_\infty\bar M$, the contact boundary at infinity, is homotopy equivalent to the space of Liouville sectors. Such objects are called \emph{Liouville pairs} or \emph{sutured Liouville manifolds}.

We will want to use several flavors of Floer theory on Liouville sectors, and to that end we define the types of Floer data we will wish to use. The definitions are routine but long, and a reader who is willing to be cavalier about compactness may safely skip them.

\begin{defn}\label{defn:Floer data admissibility}
	Let $M$ be a Liouville sector. A Hamiltonian $H$ on $M$ has \emph{growth rate $\rho$} if, outside of a compact set, it satisfies
	\[
	dH(Z)=\rho H.
	\]
	In particular, compactly supported Hamiltonians have every growth rate. Hamiltonians of growth rate $1$ or $2$ are called \emph{linear} or \emph{quadratic}, respectively. A \emph{symplectization coordinate} is an exhausting, globally linear Hamiltonian which might be non-smooth. A \emph{transverse Hamiltonian} is one which is either compactly supported or linear and nonvanishing outside a compact set.
	
	Let $\mathbf r$ be a family of symplectization coordinates over some base $\mathcal C$, and let $K\subset M$ be a compact subset. We will be interested in families of pairs of Hamiltonians over $\mathcal C$ of the form $\mathbf H=(\mathbf H^0,\mathbf H^1,\mathbf H^2)$, where $\mathbf H^i$ consists of Hamiltonians of growth rate $i$. Outside $K$, we require that these are given by
	\[
	\mathbf H^i(c)=i\cdot\bigl(\mathbf r(c)\bigr)^i\qquad c\in\mathcal C.
	\]
	
	Now suppose $\mathcal C$ is the universal curve over some compactified space of domains $\mathcal B$. In particular, $\mathcal C$ is fibered by nodal Riemann surfaces $\Sigma_b$ for $b\in\mathcal B$, and we require $d\mathbf r|_{\Sigma_b}$ to vanish on normal vectors to $\partial\Sigma_b$. We consider families of $1$-forms $\boldsymbol\beta^i\in\Gamma(\mathcal B,\Omega^1\Sigma_b)$ for $i=0,1,2$, requiring that these are appropriately sub-closed and compatible with $\mathbf r$. Sub-closedness means that 
	\begin{equation}\label{subclosedness for linear term}
	d\left(\boldsymbol\beta^1(b)\right)\le0\qquad\text{and}\qquad \boldsymbol\beta^1(b)|_{\partial\Sigma_b}=0
	\end{equation}
	outside the support of $\boldsymbol\beta^2$ and
	\[
	d\left(\boldsymbol\beta^2(b)\right)\le0\qquad\text{and}\qquad \boldsymbol\beta^2(b)|_{\partial\Sigma_b}=0
	\]
	everywhere, with the stronger inequality
	\[
	d\left(\boldsymbol\beta^2(b)\right)\le k\cdot dvol<0
	\]
	for some global $k$ on the loci where \eqref{subclosedness for linear term} fails or where $d\mathbf r|_{T\Sigma_b}\ne0$. Compatibility with $\mathbf r$ additionally means that $\mathbf r$ is independent of $\mathcal C$ on a neighborhood of the locus $\{\boldsymbol\beta^2=0\}$, and that
	\[
	d\mathbf r|_{T\Sigma_b}\wedge\boldsymbol\beta^2(b)\le0
	\]
	as $C^0(M)$-valued $2$-forms on $\Sigma_b$. Finally, we require that for all $b\in\mathcal B$, the interior of the locus $\{\boldsymbol\beta^1(b)=\boldsymbol\beta^2(b)=0\}$ is connected and contains at least one strip-like end.
	
	All almost complex structures are assumed $Z$-invariant outside $K$. On a neighborhood of the closure of $(\mathrm{supp}(\boldsymbol\beta^1)\cup\mathrm{supp}(\boldsymbol\beta^2))\times(M\setminus K)$, we additionally ask that that any $\mathcal C$-family of almost complex structures $\mathbf J$ is of \emph{contact type}, meaning that
	\[
	d\bigl(\mathbf r(c)\bigr)\circ\mathbf J(c)=-a(c)\mathbf r(c)\lambda
	\]
	for some function $a\colon\mathcal C\to\R_{>0}$ which on a neighborhood of $\mathrm{supp}(\boldsymbol\beta^1)\cup\mathrm{supp}(d\mathbf r|_{T\Sigma_b})$ is locally constant.
	
	An \emph{admissible family of Floer data} for $\mathcal C$ is a choice of $\mathbf r$, $K$, $\mathbf H$, $\mathbf J$, and $\boldsymbol\beta=(\boldsymbol\beta^0,\boldsymbol\beta^1,\boldsymbol\beta^2)$ which satisfies all the above conditions and is adapted in the usual sense to some universal choice of strip-like coordinates on $\mathcal C$.
\end{defn}

\begin{lemma}
	Let $\mathcal C$ be as in Definition \ref{defn:Floer data admissibility}, and label $\partial\Sigma_b$ with possibly-moving exact Lagrangian submanifolds of $M$ which are conical outside a compact set $K$. For any admissible Floer datum $\K$ extending $K$ such that $\mathbf H^1|_{\Sigma_b}\boldsymbol\beta^1(b)$ generates the Lagrangian deformation on $\partial\Sigma_b$ for all $b$, the space of solutions to the generalized Floer equation
	\begin{equation}\label{Floer's eqn}
	\mathcal M(\mathcal C):=\bigcup_{b\in\mathcal B}\left\{u\colon\Sigma_b\to M\,\vert\,\left(du-\Sigma(X_{\mathbf H^i}\otimes\boldsymbol\beta^i)\right)^{0,1}_\mathbf J=0\right\}
	\end{equation}
	with the given boundary conditions satisfies Gromov compactness.
\end{lemma}
\begin{proof}[Sketch of proof]
	The main step is to show that there is a compact subset of $M$ which contains the image of all elements of $\mathcal M(\mathcal C)$. For this, we cover $\mathcal C$ by open sets $U$ and $V$, where $U$ is the interior of the locus on which $\mathbf J$ is of contact type outside $K$, and $V$ is the interior of the locus $\{\boldsymbol\beta^1(b)=\boldsymbol\beta^2(b)=0\}$.
	
	On $U$, we have a pointwise maximum principle above $\mathbf r=N$ for some large $N$ \cite[expanded version]{Sylvan2019}, which means any global maximum of $\mathrm r\circ u$ which is larger than $N$ must live on the closure of the region on which $J$ is not of contact type. In particular, it must live in $V$.
	
	On the other hand, in $V$ we have a monotonicity principle as well $C^0$ estimates on $\mathbf r\circ u$ on the strip-like part \cite{GanatraPardonShende2017}, which provides a global bound on $\mathbf r\circ u$ depending only on the Floer data and the actions of the input and output chords.
\end{proof}

In the sequel, all choices of Floer data will be assumed admissible, and if $\partial\mathcal B$ is indexed by products of lower-dimensional moduli spaces, we will further assume the Floer data on $\mathcal C$ are compatible with this boundary decomposition.

Fix a ground field $\mathbb k$, which will be the coefficient field for all our Floer complexes. If $M$ is a Liouville sector, then its \emph{wrapped Fukaya category} $\W(M)$ is the Fukaya category encoding the Floer theory of exact Lagrangians in $M$ which are conical at infinity. More precisely, it is the Fukaya $A_\infty$-category whose
\begin{itemize}
	\item objects are embedded exact Lagrangian branes in $M$ which are conical at infinity, meaning they are equipped with gradings if $M$ is graded and with pin structures if you want signs, and whose
	\item morphism complexes are the Floer cochain complexes $CF^\bullet_\sigma(L_0,L_1;H)$ taken in $\bar M$, generated by those Hamiltonian chords whose linking number with $D_\sigma$ is zero, for a positive quadratic Hamiltonian $H$ whose Hamiltonian vector field is tangent to $D_\sigma$. Here, we require the almost complex structure to make $D_\sigma$ into an almost complex submanifold.
\end{itemize}

While this definition is easiest to state, it will be convenient to move between various definitions of $\mathcal W(M)$. In particular, we will want to use the definition from \cite{GanatraPardonShende2017}, which is constructed as a localized category $\W_\mathrm{loc}(M)=\O(M)/\mathcal C(M)$. When the distinction is important, we will refer to the first presentation as $\W_\mathrm{quad}(M)$.

For this, $\O$ is defined to be the $A_\infty$-category whose
\begin{itemize}
	\item objects are pairs $(L,n)$, where $L\in\Ob(\W_\mathrm{quad}(M))$ and $n\in\Z_{\ge0}$, and whose
	\item morphism complexes are given by
	\[
	\hom_\O((L_1,n_1),(L_2,n_2))=\begin{cases}
	k\cdot\mathbf{1}_{(L,n)}&n_1=n_2=n\text{ and }L_1=L_2=L\\
	CF^\bullet(\phi_{n_1}L_1,\phi_{n_2}L_2)&n_1>n_2\\
	0&\text{otherwise,}
	\end{cases}
	\]
	where $\mathbf{1}_{(L,n)}$ is a strict unit and $\{\phi_nL\}$ is a cofinal sequence of increasing positive wrappings of $L$. In this case, the almost complex structures are required to make the projection \[\pi\colon\sigma(\bar F\times\C_{0\le\Re\le1})\to\C_{0\le\Re\le1}\] holomorphic, which prevents holomorphic curves from touching $\partial M$.
\end{itemize}
The full subcategory $\mathcal C(M)\subset\Perf\,\O(M)$ consists of all cones of \emph{continuation elements}
\[
e_{(L,n)}\in\hom^0_\O((L,n+1),(L,n)),
\]
which are the Floer cocycles obtained by counting holomorphic disks for increasing wrapping in $M$.

\subsection{Continuation functors}

To justify this use of multiple definitions, we need to know that the resulting categories are quasi-equivalent in a sufficiently canonical way. We will prove this in Proposition \ref{prop:WFC are all the same} below, but first we will need a supply of functors. To this end, fix a collection of branes $\mathcal L$, and let $\K_1$ and $\K_2$ be admissible families for Floer data on the associahedra making $\mathcal L$ into object sets of some Fukaya categories $\F(M;\K_1)$ and $\F(M;\K_2)$. For a pair of objects $L,L'\in\mathcal L$, let $H_i^{L,L'}$ be such that the Floer equation for $\hom_{\F(M;\K_2)}(L,L')$ is
\[
\left(du-X_{H_i^{L,L'}}\otimes dt\right)^{0,1}=0
\]

\begin{lemma}\label{lem:continuation functors}
	Suppose that for all $L,L'$, either $H_2^{L,L'}$ is positive quadratic or the inhomogeneous terms satisfy
	\[
	H_1^{L,L'}\le H_2^{L,L'}
	\]
	outside a fixed compact subset of $M$. Then there is a \emph{``continuation''} $A_\infty$-functor
	\[
	F\colon\F(M;\K_1)\to\F(M;\K_2)
	\]
	which is canonical up to homotopy. Moreover, both $F$ and the homotopy are given by counts of perturbed holomorphic curves for admissible Floer data over certain spaces of domains.
\end{lemma}
\begin{proof}[Remark on proof]
	In the case where both $H_1^{L,L'}$ and $H_2^{L,L'}$ are positive quadratic, the construction of Floer data on the multiplihedra is given in \cite[expanded version]{Sylvan2019}. The general case is essentially identical.
	
	To pass from Floer data to maps of Floer complexes, one additionally needs coherent orientations of the multiplihedra and their boundary strata. This is treated in \cite{MauWehrheimWoodward2018}.
\end{proof}

\begin{rmk}
It is reasonable to complain that Lemma \ref{lem:continuation functors} is not applicable to $\O$, since $\O$ has an artificial strict unit. To remedy this, note that all $A_\infty$ operations of $\O$ can be viewed as coming from perturbed holomorphic disks by viewing an input of $\mathbf 1_{(L,n)}$ as a boundary marked point and pulling back the Floer data by the forgetful map. Continuation functors from $\O$ to a fully geometric Fukaya category $\F(M;\K)$ can likewise be extended to include these boundary marked points. To do this, one first defines the linear piece $F^1(\mathbf 1_{(L,n)})$ by counting disks with one boundary marked point and one output end, where the Floer datum near the marked point has vanishing $\boldsymbol\beta$. This inductively determines Floer data near the boundary of the higher dimensional multiplihedra, which we extend arbitrarily.

For a detailed account of this type of construction in a similar setting, see \cite[Chapter 10]{Ganatra2013}.
\end{rmk}

\begin{rmk}\label{rmk:operations with extra copies of objects}
A second, less significant objection is that $\O$ is not defined on a collection of branes, but rather pairs $(L,n)$. This can be accommodated with essentially no modification. Specifically, we can view the object $(L,n)$ as an extra copy of the brane $\phi_nL$ and disregard the undefined morphism spaces. Now, to define continuation functors between Fukaya categories with extra copies of branes, one chooses for each object of the source category an object of the target category supported on the same brane. From here the theory carries through verbatim.
\end{rmk}

\begin{prop}\label{prop:WFC are all the same}
	$\W_\mathrm{quad}(M)$ is canonically quasi-equivalent to the partially wrapped Fukaya category $\W_\sigma(\bar M)$ of \cite{Sylvan2019} and to $\W_\mathrm{loc}(M)$.
\end{prop}
\begin{proof}
	The Floer data used to define $\W_\sigma(\bar M)$ are a special case of the Floer data used to define $\W(M)$, so it is enough to verify invariance of $\W(M)$. This follows from the existence of continuation functors and homotopies between them, along with the observation that the subcomplex generated by those Hamiltonian chords with negative linking number with $D$ is contractible.
	
	For the equivalence between $\W_\mathrm{quad}(M)$ and $\W_\mathrm{loc}(M)$, note that all almost complex structures make $D_\sigma$ into an almost complex submanifold. By positivity of intersections, this means that the continuation functor $F\colon\O(M)\to\W_\mathrm{quad}(\bar M)$ factors through $\W_\mathrm{quad}(M)$, where we view $\O(M)$ as a full subcategory of $\O(\bar M)$. Now the induced functor $\O(M)\to\W_\mathrm{quad}(M)$ sends continuation elements to isomorphisms, so it factors through $\W_\mathrm{loc}(M)$. Write $G\colon\W_\mathrm{loc}(M)\to\W_\mathrm{quad}(M)$ for the induced functor.
	
	To see that $G$ is a quasi-equivalence, we can study just the continuation map. In this case, factor it as
	\begin{equation}\label{continuation map decomposed}
		\hom_\O(L_0,L_1)\to CF^\bullet(L_0,L_1;H_\ell^M)\to CF_\sigma^\bullet(L_0,L_1;H_\ell^{\bar M})\to \hom_{\W_\mathrm{quad}}(L_0,L_1).
	\end{equation}
	Here, $H_\ell^{\bar M}$ is a linear Hamiltonian on $\bar M$ whose Hamiltonian vector field is transverse to the rays $\sigma(F\times e^{i\theta}\R_{>0})$ and winds counterclockwise around $D_\sigma$, and $H_\ell^M=\left(\kappa H_\ell^{\bar M}\right)\big\vert_M$ for a cutoff function $\kappa$ which outside a small neighborhood $D_\sigma$ depends only on the angular coordinate of $\H$ and vanishes in a neighborhood of $\partial M$. As with $\mathcal W_\mathrm{quad}$, the notation $CF^\bullet_\sigma$ means that we take the subcomplex generated by chords whose linking number with $D_\sigma$ is zero.
	
	Now \eqref{continuation map decomposed} is compatible with increasing the linear wrapping, and taking the direct limit gives
	\[
		\hom_\O(L_0,L_1)\to \hocolim_{H_\ell^M} CF^\bullet(L_0,L_1;H_\ell^M)\to \hocolim_{H_\ell^{\bar M}} CF_\sigma^\bullet(L_0,L_1;H_\ell^{\bar M})\to \hom{\W_\mathrm{quad}}(L_0,L_1).
	\]
	Localizing by continuation elements makes the first arrow into a quasi-isomorphism \cite[Lemma 5.6]{GanatraPardonShende2017}, while the last arrow is a quasi-isomorphism by the usual argument relating linear and quadratic wrapped Floer cohomologies \cite{Ritter2013}. To see that the middle arrow is a quasi-isomorphism, note that for reasonable $\kappa$ the two complexes have the same generators and use an upper triangularity argument.
\end{proof}

\begin{rmk}
	The above proof illustrates the general theme that $\mathcal W_\mathrm{loc}$ is easy to map out of, while $\mathcal W_\mathrm{quad}$ is easy to map into.
\end{rmk}

There is a further compatibility that will be important and deserves proof.

\begin{lemma}\label{lem:acceleration is independent of target Floer data}
	Let $\K_1$ and $\K_2$ be admissible collections of Floer data presenting $\mathcal W_\mathrm{quad}$ as $\F(M;\K_1)$ and $\F(M;\K_2)$, respectively. Then the diagrams
	\[
	\begin{tikzcd}
	\O(M)\ar[r,"F_1"]\ar[dr,"F_2"] & \F(M;\K_1)\ar[d,"F"]\\
	&\F(M;\K_2)
	\end{tikzcd}
	\]
	and
	\[
	\begin{tikzcd}
	\W_\mathrm{loc}(M)\ar[r,"G_1"]\ar[dr,"G_2"] & \F(M;\K_1)\ar[d,"F"]\\
	&\F(M;\K_2)
	\end{tikzcd}
	\]
	commute up to isomorphism in the category $\Fun(\W_\mathrm{loc}(M),\F(M;\K_2))$, where $F_i$ and $F$ are continuation functors and $G_i$ are the quasi-equivalences of Proposition \ref{prop:WFC are all the same}.
\end{lemma}
\begin{proof}
	It suffices to show that the first diagram commutes up to isomorphism, since the second is obtained from it by localization. For this, follow \cite[Section 10a]{Seidel2008a} and write $\W_\mathrm{quad}^\mathrm{tot}$ for a presentation of $\mathcal W_\mathrm{quad}$ with objects $\Ob(\F(M;\K_1))\amalg\Ob(\F(M;\K_2))$ and Floer data restricting to $\K_i$ on the appropriate piece. Enlarge this further to a semi-orthogonal gluing
	\[
	\W^\rightarrow=\langle\O(M),\W_\mathrm{quad}^\mathrm{tot}\rangle,
	\]
	presented by a choice of admissible Floer data which is positive quadratic on the mixed morphism spaces.
	
	Now, given any object $L$ of $\O(M)$ and any object $L_i$ of $\F(M;\K_i)$ with the same underlying brane as $L$, we have a morphism $e_{L,L_i}\in\hom_{\W^\rightarrow}(L,L_i)$ obtained by counting holomorphic disks with one output. Write $\mathcal C^\rightarrow\subset\Perf\,\W^\rightarrow$ for the full subcategory comprising all cones of morphisms $e_{L,L_i}$. Now the composition
	\[
	I_i\colon\F(M;\K_i)\into\W^\rightarrow\to\W^\rightarrow/\mathcal C^\rightarrow
	\]
	is a quasi-equivalence by \cite[Lemma 3.13]{GanatraPardonShende2017}. On the other hand, we have continuation functors
	\begin{equation}\label{eq:factoring continuation functors}
	\begin{tikzcd}
	\W^\rightarrow\ar[rr]\ar[dr,dotted]&&\F(M;\K_i)\\
	&\W^\rightarrow/\mathcal C^\rightarrow\ar[ur,dotted,"P_i" below]
	\end{tikzcd}
	\end{equation}
	which are the identity on $F(M;\K_i)$ and which factor through $\W^\rightarrow/\mathcal C^\rightarrow$ as indicated. This implies $I_i\circ P_i\cong\mathrm{Id}_{\W^\rightarrow/\mathcal C^\rightarrow}$.
	
	Now, because a continuation functor out of $\W^\rightarrow$ restricts to a continuation functor out of $\O(M)$, \eqref{eq:factoring continuation functors} induces a factorization
	\[
	F_i\cong P_i\circ I_0,
	\]
	where $I_0$ is the composition
	\[
	\O(M)\into\W^\rightarrow\to\W^\rightarrow/\mathcal C^\rightarrow.
	\]
	This implies
	\begin{align*}
	F_2&\cong P_2\circ I_0\\
	&\cong P_2\circ I_1\circ P_1\circ I_0\\
	&\cong F\circ F_1,
	\end{align*}
	where the isomorphism $P_2\circ I_1\cong F$ again comes from restricting \eqref{eq:factoring continuation functors} to $\F(M;\K_1)\subset\W^\rightarrow$.
\end{proof}

\begin{rmk}\label{rmk:canonical homotopies}
	In fact, by working harder one could directly construct a homotopy between a composition of two continuation functors and a single continuation functor. This has been done in the world of quilts \cite{MauWehrheimWoodward2018,Bottman2017}, and the intrepid reader is invited to ``dequiltify'' them. Alternatively, one could directly implement continuation functors in a quilty framework. This would require extending Gao's work \cite{Gao2017} from the fully wrapped to the partially wrapped setting, as well as verifying compactness for pseudoholomorphic quilts deformed by Hamiltonians of mixed growth rate.
\end{rmk}

\subsection{Inclusions of sectors}

We summarize the following from \cite{GanatraPardonShende2017}.
\begin{itemize}
	\item Given a proper inclusion $i\colon M\to N$ of Liouville sectors which is conical near infinity, there is a \emph{pushforward functor}
	\[
	i_*\colon\W(M)\to\W(N)
	\]
	which is given by $i$ on objects and which comes from a fully faithful embedding $\O(M)\to\O(N)$, for some larger version of $\O(N)$. This larger version is not quasi-equivalent to the usual $\O(N)$, but it still localizes to $\W(N)$.
	\item For a chain of proper inclusions, one can arrange that the resulting diagram of pushforward functors is strictly commutative.
	\item For a \emph{trivial inclusion}, meaning that $i(M)$ is deformation equivalent to $N$, the pushforward functor $i_*$ is a quasi-equivalence.
\end{itemize}

We will need the following invariance result.

\begin{prop}\label{prop:isotopic inclusions}
	Let $i_t\colon M\to N$ be a family of inclusions of Liouville sectors as above indexed by $t\in[0,1]$. Then the pushforward functors $(i_0)_*$ and $(i_1)_*$ are isomorphic in the functor category $\Fun(\W(M),\W(N))$.
\end{prop}
\begin{rmk}
	It is easy to show that $(i_0)_*$ and $(i_1)_*$ agree up to some autoequivalence of $\W(M)$. Indeed, by breaking up $[0,1]$ into many small intervals, we can factor $i_t$ through an isotopy of trivial inclusions, so that every functor is a quasi-equivalence. The difficulty is in proving the intuitively obvious fact that the resulting autoequivalence of $\W(M)$ is trivial.
\end{rmk}
\begin{proof}
	The first and most important observation is that we can verify the proposition in any quadratic presentation of $\W(N)$. Indeed, the ``inclusion-acceleration'' functor $i_\mathrm{acc}$ given by the composition
	\begin{equation}\label{inclusion acceleration}
	\O(M)\xhookrightarrow{i}\O(N)\to\W_\mathrm{quad}(N)
	\end{equation}
	induces
	\[
	\W_\mathrm{loc}(M)\xrightarrow{i_*}\W_\mathrm{loc}(N)\to\W_\mathrm{quad}(N)
	\]
	on the subquotient, and by Lemma \ref{lem:acceleration is independent of target Floer data} the $i_\mathrm{acc}$ is preserved by continuation equivalences of $\W_\mathrm{quad}(N)$. In fact, we will prove the stronger fact that $(i_0)_\mathrm{acc}\cong(i_1)_\mathrm{acc}$.
	
	We now perform some preliminary geometric manipulations. By postcomposing with a trivial inclusion, we can enlarge $N$ so that $i_t(\partial M)$ is disjoint from $\partial N$ for all $t$. That done, we can extend $i_t$ to a family of Liouville automorphisms $\bar i_t\colon\bar N\to\bar N$ starting at $i_0=\mathrm{id}_{\bar N}$ which is trivial on the stop $\sigma_N$. Such a family is automatically given by the Hamiltonian flow of a time-dependent linear Hamiltonian which vanishes on $\sigma_N$. By \cite[expanded version Lemma 2.17]{Sylvan2019}, $\bar i_t$ can be approximated rel endpoints in $C^0$ by the flow of a time-dependent transverse Hamiltonian $H_t$ whose flow $\phi_t$ fixes $D_{\sigma_N}$.
	
	Fix a background presentation $\mathcal P$ of $\W_\mathrm{quad}(N)$ given by Floer data $\K$ which receives $(i_0)_\mathrm{acc}$. We will construct a new presentation $\mathcal P^+$ which receives the functors $(i_0)_\mathrm{acc}$ and $(i_1)_\mathrm{acc}$ simultaneously. Our objects set will be $\mathcal L\amalg\mathcal L_1$, where $\mathcal L$ are the objects of $\mathcal P$ and $\mathcal L_1$ are additional copies of each object geometrically supported in $i_1(M)$. The Floer data are taken to agree with $\K$ on $\mathcal L$ and with the pushforward $(\phi_1)_*\K$ on $\mathcal L_1$. As functors to $\mathcal P^+$, $(i_0)_\mathrm{acc}$ is as before, and $(i_1)_\mathrm{acc}$ is the functor determined by the same Floer data pushed forward to land in $\mathcal L_1$.
	
	Now transverse Hamiltonians are precisely those which appear as the linear part of admissible Floer data (with moving boundary conditions), so Seidel's construction of natural transformations \cite[Section 10c]{Seidel2008a} carries through to give an isomorphism of functors
	\[
	T\colon\mathrm{Id}_{\W_\mathrm{quad}^\mathrm{free}(N)}\to F^{\phi_1},
	\]
	where $\W_\mathrm{quad}^\mathrm{free}(N)$ is an enlargement of $\W_\mathrm{quad}(N)$ which containing a copy of $\mathcal P$ for each automorphism $\phi$ of $(\bar N,D_{\sigma_N})$, where the Floer data on that copy are pushed forward by $\phi$. The functor $F^{\phi_1}$ is the one coming from the tautological $\mathrm{Aut}(\bar N,D_{\sigma_N})$ action. Now $\mathcal P^+$ is naturally a full subcategory of $\W_\mathrm{quad}^\mathrm{free}$, and $F^{\phi_1}$ sends objects of $\mathcal L$ supported on the image of $\phi_0$ into $\mathcal L_1$. This means $T$ induces the desired isomorphism
	\[
	(i_0)_*\to F^{\phi_1}\circ(i_0)_*=(i_1)_*.
	\]
\end{proof}

There is a special type of inclusion with respect to which Floer theory is particularly well-behaved.

\begin{defn}\cite{GanatraPardonShende2018}
	An inclusion $i\colon M\to N$ of Liouville sectors is \emph{forward stopped} if there is some compact, codimension zero submanifold $W\subset\overline{\partial_\infty N\setminus\partial_\infty M}$ with corners, together with a Reeb vector field $R$ on $\partial_\infty N$, such that
	\begin{enumerate}
		\item $R$ points out of $\partial_\infty M$ along $\partial^+\partial_\infty M$ and into $\partial_\infty M$ along $\partial^-\partial_\infty M$), and similarly with $M$ replaced by $N$.
		\item $\partial W$ is a union of strata $\partial^+W\cup\partial^-W\cup\partial_{\mathrm{ne}}W$ meeting along corners.
		\item $R$ points weakly into $W$ along $\partial_{\mathrm{ne}}W$.
		\item $\partial^-W=\partial^+\partial_\infty M$.
		\item $\partial^+W$ is a Liouville subdomain of $\partial^+\partial_\infty N$.
	\end{enumerate}
	
	When we need to be more explicit, we will say $i$ is forward stopped \emph{by} $W$ \emph{into} $\sigma$, where $\sigma$ is any stop of $N$ containing $\partial^+W$.
\end{defn}
Intuitively, this just says we can find a Reeb vector field whose trajectories do not exit and reenter $M$. Note that our definition is slightly stricter than the original, since we have required compatibility with a specified decomposition of the convex surfaces $\partial_\infty M$ and $\partial_\infty N$.

\begin{prop}\label{prop:forward stopped inclusions}
	Let $i\colon M\to N$ be forward stopped by $W$. Then
	\begin{enumerate}
		\item\label{item:forward stopped ff} $i_*$ is fully faithful.
		\item\label{item:forward stopped disjoint} If $L\in\W(N)$ is disjoint from $M$ and $\partial_\infty L$ is disjoint from $W$, then there are Floer data presenting $\W(N)$ such that $\hom_{\W(N)}(i_*X,L)=0$ for all $X\in\W(M)$.
	\end{enumerate}
\end{prop}
\begin{proof}
	\eqref{item:forward stopped ff} is \cite[Corollary 8.7]{GanatraPardonShende2018}.
	
	\eqref{item:forward stopped disjoint}: A wrapping Hamiltonian whose flow is parallel to $R$ will have no trajectories from $i_*X$ to $L$.
\end{proof}

\subsection{The Orlov functor}

% Allow components of stops. Define stab and Orlov functor. Define stop removal.

Before we continue, note that we have not required our stops to be connected. Indeed, we will make heavy use of disconnected stops and typically manipulate them one component at a time. To make this more convenient, we will overload notation and say that a stop \emph{of} a Liouville sector $M=\bar M\setminus\sigma$ is a union of connected components of $\sigma$, or equivalently a union of boundary components of $M$. If $\sigma_j$ is a stop of $M$, then the Liouville sector
\[
M\cup\sigma_j=\bar M\setminus(\sigma\setminus\sigma_j)
\]
still has boundary unless $\sigma_j=\sigma$.

If $\bar F$ is a Liouville manifold, then its \emph{stabilization} $\Sigma\bar F$ is the Liouville sector $\bar F\times T^*[0,1]$. Stabilization induces a fully faithful functor
\[
\Sigma_*\colon\W(\bar F)\to\W(\Sigma\bar F)
\]
which sends an object $L\in\W(\bar F)$ to a conicalization of $L\times T^*_\frac12[0,1]$ \cite{GanatraPardonShende2018}. $\Sigma_*$ depends on a choice of grading on $T^*_\frac12[0,1]$, which we fix once and for all.

If $\sigma$ is a stop of $M$ with fiber $\bar F$, then we have an inclusion of Liouville sectors $i\colon\Sigma\bar F\to M$ given by a conicalization of the symplectomorphism
\[
T^*[0,1]\cong\C_{0\le\Re\le1}.
\]
The \emph{Orlov functor} is the composition
\[
\imath_\sigma=i_*\circ\Sigma_*\colon\W(\bar F)\to\W(M).
\]
This too depends on a choice of graded lift for $i$, which cannot be done canonically. This choice will generally be left implicit as it will not substantially impact our arguments, but when it is important we will draw attention to it. Typically, it will manifest in the grading shifts we use to define various operations.

\begin{rmk}
	The term ``conicalize'' does not have a precise definition, but we will use it throughout the text to mean ``modify near infinity to be compatible with the Liouville flow''. The arguments involving conicalization are usually not sensitive to the precise formula for this, provided it is not unnecessarily complicated and the cutoff functions do not zigzag.
\end{rmk}

\begin{lemma}\label{lem:Orlov deformation invariance}
Let $M_t=\bar B\setminus\sigma_t$ be a family of Liouville sectors parametrized by $t\in[0,1]$, where $\sigma_t$ has fiber $\bar F_t$. Write $\phi\colon\W(\bar F_0)\to\W(\bar F_1)$ for the quasi-equivalence coming from the Liouville automorphism $\bar F_0\cong\bar F_1$ supplied by Moser's lemma, and write $\Phi\colon\W(M_0)\to\W(M_1)$ for the quasi-equivalence coming from deformation invariance, which is realized by a zigzag of trivial inclusions. Then
\begin{equation}\label{Orlov deformation invariance}
\imath_{\sigma_1}\circ\phi\cong\Phi\circ\imath_{\sigma_0}.
\end{equation}
\end{lemma}
\begin{proof}
We restrict to the case where deformation invariance is induced by a single zigzag
\[
M_0\hookleftarrow M^\mathrm{small}\into M_1,
\]
where the fiber of the stop $\sigma^\mathrm{small}$ of $M^\mathrm{small}$ is $F^0$.
The Liouville isomorphism $\bar F_0\to\bar F_1$ extends to a proper embedding $\Sigma\bar F_0\into\Sigma\bar F_1$, and the left hand side of \eqref{Orlov deformation invariance} comes from stabilizing and applying the composition
\[
\Sigma\bar F_0\into\Sigma\bar F_1\into M_1.
\]
It is generally not true that this embedding preserves the product decomposition, but it is still the case that it is isotopic to the inclusion $\Sigma\bar F_0\into M^\mathrm{small}$, which is itself isotopic to $\Sigma\bar F_0\into M_0$. By Proposition \ref{prop:isotopic inclusions}, the corresponding functors are isomorphic.
\end{proof}

\begin{defn}
	Say that a Liouville manifold $F$ \emph{satisfies stop removal} if, for any stop $\sigma$ with fiber $\bar F$ of any Liouville sector $M$, the functor
	\[
	\mathcal{SR}\colon\W(M)/\B_\sigma\to\W(M\cup\sigma)
	\]
	is fully faithful, where $\B_\sigma\subset\W(M)$ is the full subcategory of objects in the image of $\imath_\sigma$ and $\mathcal{SR}$ is induced by the inclusion functor
	\[
	i_*\colon\W(M)\to\W(M)/\B_\sigma\xrightarrow{\mathcal{SR}}\W(M\cup\sigma).
	\]
\end{defn}

\begin{eg}
	Every Weinstein manifold satisfies stop removal.
	
	Possibly more generally, any Liouville manifold admitting a singular Lagrangian spine and such that each open stratum of the spine has a transverse ``co-core disk'' satisfies stop removal \cite{GanatraPardonShende2018}.
	
	In another direction, any Liouville domain which (1) satisfies Abouzaid's generation criterion \cite{Abouzaid2010}, (2) has nonzero symplectic cohomology, and such that (3) the preimage under $\mathcal{OC}$ of $\mathbf{1}\in SH^\bullet(F)$ has zero action satisfies stop removal \cite{Sylvan2019}. By performing this proof in a setting with the horizontal confinement principle of \cite{GanatraPardonShende2017}, one expects to be able to show that Abouzaid's criterion alone suffices.
\end{eg}

If $\bar F$ satisfies stop removal, then it is immediate that
\[
\Perf\,\Sigma_*\colon\Perf\,\W(\bar F)\to\Perf\,\W(\Sigma\bar F)
\]
is a quasi-equivalence.

\begin{const}\label{const:orlov via ham}
We now give an alternative construction of the Orlov functor which is better adapted to morphism complexes defined using Hamiltonians.

Again let $M$ be a Liouville sector, and let $\sigma$ be a stop of $M$ with fiber $\bar F$. Let $i\colon\Sigma\bar F\to M$ be the corresponding inclusion of Liouville sectors, and choose the Liouville form $\lambda$ on $\bar M$ to strictly extend the product Liouville form on $\Sigma\bar F$. This costs nothing, since the definition of $i$ already makes $\lambda$ agree with the product form outside a compact set.

Let $H_\Sigma\colon\Sigma\bar F\to\R$ be a quadratic conicalization of $H_{\bar F}+\lvert p\rvert^2$, where $p$ is a momentum coordinate on $T^*[0,1]$ and $H_{\bar F}$ a positive quadratic Hamiltonian on $\bar F$. For reasonable $H_{\bar F}$, this satisfies
\begin{equation}\label{H on stab}
\frac\partial{\partial\lvert p\rvert}H_\Sigma\ge2\lvert p\rvert.
\end{equation}
Here, ``reasonable'' depends on the precise conicalization formula -- the usual formula for conicalizing a split quadratic Hamiltonian yields the condition $Z_{\bar F}H_{\bar F}\ge2H_{\bar F}$.

Extend $H_\Sigma$ to a quadratic Hamiltonian $H$ on $\bar M$. Given a brane $L\subset\hat F$ admitting a compactly supported primitive, choose its stabilization $\Sigma_*L$ to be split over a large subset of $T^*[0,1]$, so that any time 1 trajectory of $X_H$ starting on the non-split part of $i(\Sigma_*L)$ must leave the image of $i$.

Now, the wonderful thing about this setup is that $\bar M$ admits a contact type almost complex structure $J$ which, for all $q\in[0,1]$, makes $\bar F\times\{q\}$ into an almost complex submanifold. Because $X_H$ is also tangent to each of these submanifolds, $H$-perturbed holomorphic curves satisfy positivity of intersection with each of them. Together with asymptotic analysis near the output of a holomorphic disk, this implies that the span of the Floer generators in $CF^\bullet_\sigma(i(\Sigma_*L_0),i(\Sigma_*L_1);H)$ which geometrically live in $\bar F\times\{(p,q)=(0,\frac12)\}$ form a subcomplex $C$, that this subcomplex is closed under all $A_\infty$-operations, and that the corresponding holomorphic disks all live in $\bar F\times\{(0,\frac12)\}$. These facts canonically identify $C$ with $CF^\bullet(L_0,L_1,H_{\bar F})$. We can then define the restricted Orlov functor $\imath_\sigma|_{\W_0(\bar F)}$ to be the inclusion of $C$ on each morphism space, where $\W_0(\bar F)$ is the full subcategory of $\W(\bar F)$ of branes admitting compactly supported primitives. Because every exact Lagrangian is isotopic to one admitting a compactly supported primitive, this is essentially as good as all of $\imath_\sigma$.

For presentations of wrapped Floer cohomology using linear Hamiltonians, the same construction works after replacing $H$ with large multiples of $\sqrt H$.
\end{const}

\subsection{The Viterbo sector}

We recall the sectorial construction of the Viterbo transfer map \cite[Section 8.3]{GanatraPardonShende2018}. Let $\bar M$ be a Liouville manifold and $M^\mathrm{in}\subset\bar M$ a Liouville subdomain, i.e. a compact, codimension zero submanifold with boundary such that $Z$ points out along $\partial M^\mathrm{in}$. The boundary at infinity of $\bar M\times\C$ has a subset contactomorphic under Liouville flow to $\bar M\times S^1$. Under this identification, define the \emph{Viterbo sector} $V_{M^\mathrm{in}}$ to be the Liouville sector associated to the Liouville pair
\[
(\bar M\times\C, \sigma_0\cup\sigma_1),
\]
where $\sigma_0$ is the closure of $\bar M\times\{-1\}$ and $\sigma_1=M^\mathrm{in}\times\{1\}$. This gives rise to Orlov functors $\imath_{\sigma_j}$ associated to the inclusions $i_j\colon\Sigma\bar F_{\sigma_j}\to V_{M^\mathrm{in}}$. We grade these inclusions so that
\[
i_0|_{M^\mathrm{in}\times T^*[0,1]}\simeq i_1|_{M^\mathrm{in}\times T^*[0,1]}[1]
\]
as graded symplectic embeddings, where as usual the shift $[1]$ indicates a shift \emph{down} by one.

\begin{rmk}
	While the precise value of this shift is a matter of convention, an even shift would be slightly unnatural, since it is not compatible with actual orientations on Lagrangians. To be more explicit, note that on objects, the pushforward functors are induced by graded symplectic embedding. This holds even for the orientation $\Z/2$-grading, which means that when we interpret the gradings as actual orientations, they are not automatically intertwined by the diffeomorphism of Lagrangians. The odd shift ensures that they are.
	
	We will proceed with these more natural conventions at the cost of messier formulas. To the reader who would prefer to dispense with this compatibility, I apologize.
\end{rmk}

Now suppose $M^\mathrm{in}$ (or rather its completion) satisfies stop removal. Then 
\[
\mathcal{SR}\colon\W(V_{M^\mathrm{in}})/\B_{\sigma_1}\to\W(V_{M^\mathrm{in}}\cup\sigma_1)=\W(\bar M\times\C\setminus\sigma_0)\cong\mathbf 0
\]
is fully faithful, so the image of $\imath_{\sigma_1}$ split-generates $\W(V_{M^\mathrm{in}})$. On the other hand, Proposition \ref{prop:forward stopped inclusions} implies that $\imath_{\sigma_1}$ is fully faithful. Together, this shows that $\Perf\,\imath_{\sigma_1}$ is a quasi-equivalence, and the \emph{Viterbo transfer map} is defined to be the composition
\[
\mathcal V=(\Perf\,\imath_{\sigma_1}[1])^{-1}\circ\imath_{\sigma_0}\colon\W(\bar M)\to\Perf\,\W(\bar M^\mathrm{in}).
\]

\begin{rmk}
	By proving a stronger stop removal statement, Ganatra, Pardon, and Shende in fact construct the Viterbo transfer map at the level of $\mathrm{Tw}$ instead of $\Perf$ \cite{GanatraPardonShende2018}. However, because we cannot avoid passing to modules in the proof of Theorem \ref{thm:Viterbo is epic}, we will not concern ourselves with the differences at this point.
\end{rmk}

\begin{prop}
	If $L\in\W(\bar M)$ is supported on a compact Lagrangian in $M^\mathrm{in}$, then $\mathcal V(L)$ is the corresponding brane in $\W(\bar M^\mathrm{in})$, and on morphisms
	\[
	\hom_{\W(M)}(\,\cdot\,,L)\to\hom_{\Perf\W(\bar M^\mathrm{in})}(\mathcal V(\,\cdot\,),\mathcal V(L))
	\]
	and
	\[
	\hom_{\W(M)}(L,\,\cdot\,)\to\hom_{\Perf\W(\bar M^\mathrm{in})}(\mathcal V(L),\mathcal V(\,\cdot\,))
	\]
	are quasi-isomorphisms of $\left(\mathrm{end}(L),\W(M)\right)$ and $\left(\W(M),\mathrm{end}(L)\right)$ bimodules.
\end{prop}
\begin{proof}[Sketch of proof]
	The claim about objects follows from the fact that $\imath_{\sigma_0}L$ and $\imath_{\sigma_1}L[1]$ are Hamiltonian isotopic in $V_{M^\mathrm{in}}$.
	
	For the claims about morphisms, let $U\subset M^\mathrm{in}$ be a Weinstein neighborhood of $L$. We can arrange that wrapping $\imath_{\sigma_0}L$ backwards or forwards in $\Sigma\bar M\subset V_{M^\mathrm{in}}$ crosses $\partial\Sigma\bar M$ in the Liouville cone over $U\times S^1$. From there, it is easy to prevent return to $\Sigma\bar M$, as in the theory of forward stopped inclusions. 
\end{proof}

\subsection{Comparison with the Abouzaid--Seidel construction}\label{sec:abouzaid--seidel}

Suppose now that $M^\mathrm{in}\subset\bar M$ is a Liouville subdomain satisfying stop removal, and that $\W_0(\bar M)\subset\W(\bar M)$ is the full subcategory consisting of exact Lagrangians $L$ which are conical in a neighborhood of $\partial M^\mathrm{in}$ and which admit a primitive vanishing on $L\cap\partial M^\mathrm{in}$. For these ``strongly exact'' Lagrangians, Abouzaid and Seidel \cite{AbouzaidSeidel2010} construct a Viterbo transfer map $\mathcal V_\mathrm{AS}$ by a neck-stretching procedure modeled after Viterbo's construction for symplectic cohomology \cite{Viterbo1999}. We wish to see that $\mathcal V_\mathrm{AS}$ agrees with the restriction $\mathcal V|_{\W_0(\bar M)}$.

That said, we also wish to avoid delving into the details of Abouzaid and Seidel's construction, so instead we list its essential properties.

\begin{itemize}
	\item The ambient wrapped Fukaya categories are presented using (linear) Hamiltonians.
	\item On morphisms, $\mathcal V_\mathrm{AS}$ is defined using counts of perturbed holomorphic curves in neck-stretched copies of $\bar M$ (this is of course Liouville isomorphic to $\bar M$ itself, but using this fact would require keeping track of $\partial M^\mathrm{in}$).
	\item If $\bar M\setminus M^\mathrm{in}$ is a trivial cobordism, then $\mathcal V_\mathrm{AS}$ is a quasi-equivalence.
\end{itemize}

A side benefit of taking this approach is that it applies equally well to any other construction of the Viterbo transfer functor using Hamiltonians. In particular, our argument will show that any extension of $\mathcal V_\mathrm{AS}$ to non-strongly exact Lagrangians using a Maurer-Cartan term which comes from counting holomorphic curves will still be isomorphic to our $\mathcal V$.

To begin, note that the wrapped Fukaya category presented with linear Hamiltonians $\W_\mathrm{lin}(M)$ agrees with the other versions of wrapped Fukaya category we have considered. Let us sketch the proof. Begin by building a larger linear category $\W_\mathrm{lin}^+(M)$ whose objects are pairs $(L,n)$, where $n\in\Z_{\ge0}$, and whose morphism complexes $\hom\left((L_0,n_0),(L_1,n_1)\right)$ have an initial $H=0$ piece whenever $n_0>n_1$. After adjoining a strict unit, $\W_\mathrm{lin}^+(M)$ admits tautological inclusions from $\W_\mathrm{lin}(M)$ and $\O$, where the first inclusion is a quasi-equivalence and the second becomes a quasi-equivalence after localizing by continuation elements, as in the proof of Proposition \ref{prop:WFC are all the same}. This works not only on ordinary Liouville manifolds, but just as well on Liouville sectors, if the Floer data is chosen compatibly with $D_\sigma$ and the generators are restricted to those whose linking number with $D_\sigma$ is zero.

The neck-stretching procedure of \cite{AbouzaidSeidel2010} works essentially without modification for $\W_\mathrm{lin}$ or $\W_\mathrm{lin}^+$, defining a transfer map between Liouville sectors. Spelling this out and suppressing the ${}_\mathrm{lin}$ or ${}_\mathrm{lin}^+$, define a \emph{sectorial Liouville subdomain} $M^\mathrm{in}$ of a Liouville sector $M$ to be a compact, codimension zero submanifold with boundary and corners, such that
\begin{itemize}
	\item the boundary comes in vertical $\partial^vM^\mathrm{in}\subset\partial M$ and horizontal $\partial^hM^\mathrm{in}\pitchfork\partial M$ flavors,
	\item all corners lie in $\partial M$,
	\item $Z$ points outward along $\partial^hM^\mathrm{in}$, and
	\item attaching the positive half of the symplectization of $\partial^hM^\mathrm{in}$ makes $M^\mathrm{in}$ into a Liouville sector.
\end{itemize}

Given a sectorial Liouville subdomain $M^\mathrm{in}$ of $M$, we can extend it to an ordinary Liouville subdomain $M^\mathrm{in}_+$ of $\bar M$ without changing its intersection with $M$. It follows that, up to attaching symplectizations, the stop $\sigma^\mathrm{in}$ of $M^\mathrm{in}$ has fiber $F^\mathrm{in}\cong D_\sigma\cap M^\mathrm{in}_+$, and that $M^\mathrm{in}\cup\sigma^\mathrm{in}\cong M^\mathrm{in}_+$.

Now consider stretching the neck $\bar M$ along $\partial M^\mathrm{in}_+$. Choose Floer data for this which extend Floer data for $\bar F$ stretched along $\partial F^\mathrm{in}$. Then all resulting holomorphic curves will have positive intersections with $D_\sigma$, so by a winding number argument the functor $\mathcal V_\mathrm{AS}\colon\W_0(\bar M)\to\W(\bar M^\mathrm{in}_+)$ will induce a functor $\mathcal V_\mathrm{AS}^\mathrm{sect}\colon\W_0(M)\to\W(\bar M^\mathrm{in})$.

To compare the two constructions of the Viterbo map, we will study the sectorial Liouville subdomain $V^\mathrm{in}\subset V_{M^\mathrm{in}}$ given by
\[
V^\mathrm{in}=V_{M^\mathrm{in}}\cap\left(\bar M^\mathrm{in}\times\C\right),
\]
where the intersection takes place in $\bar M\times\C$. After attaching symplectizations, we have $V^\mathrm{in}\cong\Sigma M^\mathrm{in}$ as Liouville sectors, so $\Perf\,\imath_{\sigma_0^\mathrm{in}}$ and $\Perf\,\imath_{\sigma_1^\mathrm{in}}$ are quasi-equivalences, where $\sigma_0^\mathrm{in}$ and $\sigma_1^\mathrm{in}$ are the stops of $V^\mathrm{in}$. Putting all of this together, we have a diagram

\[
\begin{tikzcd}
\W(\bar M) \ar[d,"\mathcal V_\mathrm{AS}"] \ar[r,"\imath_{\sigma_0}"] \ar[rr, bend left, "\mathcal V"] &
\Perf\,\W(V_{M^\mathrm{in}}) \ar[d,"\mathcal V_\mathrm{AS}^\mathrm{sect}"] &
\Perf\,\W(\bar M^\mathrm{in}) \ar[d,"\mathcal V_\mathrm{AS}" right,"\cong" left] \ar{l}[above]{\imath_{\sigma_1}[1]}[below]{\cong}
\\
\Perf\,\W(\bar M^\mathrm{in}) \ar[r,"\imath_{\sigma_0^\mathrm{in}}" above,"\cong" below] \ar[rr, bend right, "\mathrm{Id}"]&
\Perf\,\W(V^\mathrm{in}) &
\Perf\,\W(\bar M^\mathrm{in}) \ar{l}[above]{\imath_{\sigma_1^\mathrm{in}}[1]}[below]{\cong}
\end{tikzcd}
\]

By constraining holomorphic curves as in Construction \ref{const:orlov via ham}, we see that the rectangles commute on the nose, and the triangles commute up to isomorphism. This implies that $\mathcal V$ and $\mathcal V_\mathrm{AS}$ agree up to applications of $\mathcal V_\mathrm{AS}$ on trivial cobordisms.

\section{Subdomains and bimodules}\label{ch:subdomains and bimodules}

% Review gluing. Prove that directed gluing is tensor product of bimodules. State and prove localization theorem by showing that glued sector is deformation equivalent to Viterbo sector with doubled little stop. 

\subsection{Algebraic preliminaries}
% Reference for bimodules (use standard backwad composition). Diagonal, graph Delta(dot, F dot) & adjoint graph Delta(F dot, dot) bimodules. Homological epimorphisms and localizations. Split generation. State classical commutativity and conjecture (?) commutativity
\subsubsection{Bimodules}
We work with \emph{reduced little-endian signs} and \emph{left composition}. ``Reduced'' means that morphisms in $A_\infty$-categories have reduced degree when computing Koszul signs. ``Little-endian'' means that operators act from the source side, which is usually indexed by the smaller numbers. ``Left composition'' means we write morphisms target-to-source, as in function composition
\[
\mu^2(f,g)(x)=f\circ g(x)=f(g(x)).
\]
In other words, we follow conventions where the $A_\infty$ operations $\mu^d$ satisfy
\[
\sum_{i=0}^{d-1}\sum_{k=1}^{d-i}(-1)^{(\lvert a_1\rvert-1)+\dotsm+(\lvert a_i\rvert-1)}\mu\left(a_d,\dotsc,a_{i+k+1},\mu^k\left(a_{i+k},\dotsc,a_{i+1}\right),a_i,\dotsc,a_1\right).
\]
When working with $A_\infty$-bimodules, we similarly follow the sign conventions of \cite{Ganatra2013}, which agree with those of \cite{Seidel2008} up to reading in a mirror. I refer the reader to those references for detailed definitions and recall here only the most basic idea of the relevant concepts.

Explicitly, for $A_\infty$-categories $\mathcal A$ and $\mathcal B$, an $(\mathcal A,\mathcal B)$-bimodule $M$ consists of graded vector spaces $M(A,B)$ for each $A\in\mathcal A$ and $B\in\mathcal B$, together with structure maps
\begin{align*}
\mu^{k\vert1\vert l}\colon\hom_{\mathcal A}(A_{k-1},A_k)\otimes\dotsm\otimes\hom_{\mathcal A}(A_0,A_1)\otimes M(A_0,B_0)&\otimes\hom_{\mathcal B}(B_1,B_0)\otimes\dotsm\otimes\hom_{\mathcal B}(B_l,B_{l-1})\\
&\to M(A_k,B_l)
\end{align*}
of degree $1-k-l$ satisfying a family of lengthy identities.

Given functors $F\colon\mathcal A'\to\mathcal A$ and $G\colon\mathcal B'\to\mathcal B$, we can form the \emph{pullback} $(\mathcal A',\mathcal B')$-bimodule $(F,G)^*M$, which at the level of objects associates to $(A',B')$ the space $M(F(A'),G(B'))$.

We can also take \emph{tensor products}: for $M$ an $(\mathcal A,\mathcal B)$-bimodule and $N$ a $(\mathcal B,\mathcal C)$-bimodule, their tensor product is a $(\mathcal A,\mathcal C)$-bimodule $M\otimes_{\mathcal B}N$ with
\[
M\otimes_{\mathcal B}N(A,C)=\sum_{B_0,\dotsc,B_k\in\mathcal B}M(A,B_k)\otimes\hom_{\mathcal B}(B_{k-1},B_k)[1]\otimes\dotsm\otimes\hom_{\mathcal B}(B_0,B_1)[1]\otimes N(B_0,C).
\]

Of particular importance is the \emph{diagonal} $(\mathcal A,\mathcal A)$-bimodule $\Delta_{\mathcal A}$, which at the level of objects has
\[
\Delta_{\mathcal A}(X_1,X_0)=\hom_{\mathcal A}(X_0,X_1).
\]
Up to isomorphism, the diagonal bimodule is a unit for the tensor product. For a functor $F\colon\mathcal A\to\mathcal B$, we have the \emph{graph} $(\mathcal A,\mathcal B)$-bimodule
\[
\Gamma(F):=(F,\mathrm{Id}_{\mathcal B})^*\Delta_{\mathcal B}
\]
and \emph{adjoint graph} $(\mathcal B,\mathcal A)$-bimodule
\[
\Gamma^\dagger(F):=(\mathrm{Id}_{\mathcal B},F)^*\Delta_{\mathcal B}.
\]
Either of these completely encodes $F$ up to isomorphism.

Given an $(\mathcal A,\mathcal B)$-bimodule $M$, we can form the \emph{semiorthogonal gluing} of $\mathcal A$ and $\mathcal B$ along $M$. This is an $A_\infty$-category $\mathcal C$ with
\[
\Ob(\mathcal C)=\Ob(\mathcal B)\amalg\Ob(\mathcal A)
\]
and
\[
\hom_{\mathcal C}(X,Y)=\begin{cases}
\hom_{\mathcal B}(X,Y)&X,Y\in\mathcal B\\
\hom_{\mathcal A}(X,Y)&X,Y\in\mathcal A\\
M(Y,X)&X\in\mathcal B,Y\in\mathcal A\\
0&X\in\mathcal A,Y\in\mathcal B.
\end{cases}
\]
The $A_\infty$ operations on each component agree with those of $\mathcal A$ or $\mathcal B$, while those on mixed morphisms spaces are given by
\[
\mu^{k+l+1}(a_k,\dotsc,a_1,m,b_1,\dotsc,b_l)=(-1)^{(\lvert b_1\rvert-1)+\dotsm+(\lvert b_l\rvert-1)+1}\mu^{k\vert1\vert l}(a_k,\dotsc,a_1,m,b_1,\dotsc,b_l).
\]
The sign comes from the fact that the identities satisfied by the bimodule structure maps use the unreduced degree of $m$ instead of the reduced degree, and it is the same sign as appears in the definition of $\Delta$. We will write $\mathcal C=\langle\mathcal B,\mathcal A\rangle$ when the bimodule is understood and
\[
\mathcal C=\begin{vmatrix}
\mathcal B&M\\&\mathcal A
\end{vmatrix}
\]
when it otherwise isn't. In the notation of \cite{Tabuada2007}, this is the same as the totalization of the upper triangular category $
\begin{pmatrix}
\mathcal B&M\\0&\mathcal A
\end{pmatrix}$.

\begin{lemma}[{\cite{Tabuada2007}, see also \cite[Proposition 7.7]{KuznetsovLunts2014}}]\label{prop:K-L}
	An $A_\infty$-functor $F\colon\begin{vmatrix}
	\mathcal B&M\\&\mathcal A
	\end{vmatrix}\to\mathcal D$ is precisely the data of
	\begin{itemize}
		\item an $A_\infty$-functor $F_{\mathcal B}\colon\mathcal B\to\mathcal D$,
		\item an $A_\infty$-functor $F_{\mathcal A}\colon\mathcal A\to\mathcal D$, and
		\item a closed morphism of $(\mathcal A,\mathcal B)$-bimodules $F_M\in\hom^0\left(M,(F_{\mathcal A},F_{\mathcal B})^*\Delta_{\mathcal D}\right)$.
	\end{itemize}
\end{lemma}
We will also need three-term semiorthogonal gluings. These can be described in similar terms to the above or viewed as iterated two-term gluings.

\subsubsection{Localizations and epimorphisms}
Let $\mathcal C\subset\mathcal A$ be a full $A_\infty$ subcategory. Write $\mathcal A/\mathcal C$ for the Lyubashenko--Ovsienko quotient category \cite{LyubashenkoOvsienko2006} with $\Ob(\mathcal A/\mathcal C)=\Ob(\mathcal A)$ and
\[
\hom_{\mathcal A/\mathcal C}(X,Y)=\bigoplus_{j\ge0}\bigoplus_{C_1,\dotsc,C_j\in\mathcal C}\hom_{\mathcal A}(C_j,Y)\otimes\dotsm\otimes\hom_{\mathcal A}(X,C_1)[j],
\]
where by convention the $j=0$ term is just $\hom_{\mathcal A}(X,Y)$. The $A_\infty$ operations are given by
\begin{align*}
\mu_{\mathcal A/\mathcal C}^d(a_d^{j_d}\otimes\dotsm\otimes &a_d^0\,,\,\dotsc\,,\,a_1^{j_1}\otimes\dotsm\otimes a_1^0)\\
&=\sum_{k=0}^{j_d}\sum_{l=0}^{j_1}(-1)^{(\lvert a_1^0\rvert-1)+\dotsm+(\lvert a_1^{l-1}\rvert-1)}a_d^{j_d}\otimes\dotsm\otimes a_d^{k+1}\otimes\mu_{\mathcal A}(a_d^k,\dotsc,a_1^l)\otimes a_1^{l-1}\otimes\dotsm\otimes a_1^0.
\end{align*}
In general, a \emph{localization} is a functor
\[
F\colon\mathcal A\to\mathcal B
\]
such that the induced map
\[
\Perf\,\bigl(\mathcal A/(\ker F)\bigr)\to\Perf\,\mathcal B
\]
is a quasi-equivalence.

There is a related but weaker notion, which is what we aim to show the Viterbo transfer map satisfies.

\begin{defn}\label{def:homological epi}
	A functor $F\colon\mathcal A\to\mathcal B$ is a \emph{homological epimorphism} if it satisfied one of the following three equivalent conditions.
	\begin{enumerate}
		\item The extension of scalars functor $F_!=\bullet\otimes\Gamma(F)\colon\mathrm{Mod}\text{-}\mathcal A\to\mathrm{Mod}\text{-}\mathcal B$ is a localization.
		\item The pullback functor $F^*\colon\mathrm{Mod}\text{-}\mathcal B\to\mathrm{Mod}\text{-}\mathcal A$ is fully faithful.
		\item The natural map $\Gamma^\dagger(F)\otimes_{\mathcal A}\Gamma(F)\to\Delta_{\mathcal B}$ is a quasi-isomorphism of $(\B,\B)$-bimodules.
	\end{enumerate}
\end{defn}

The reader can find an account of this equivalence and other useful background on homological epimorphisms in \cite{Efimov2013}. For our purposes, we recall two facts and state one conjecture as motivation for this definition.

\begin{thm}[\cite{Neeman1996}, see also \cite{Efimov2013}]\label{thm:Neeman}
	For an $A_\infty$-functor $F\colon\mathcal A\to\mathcal B$, the following are equivalent.
	\begin{enumerate}
		\item $F$ is a localization.
		\item $F$ is a homological epimorphism, and $\ker F_!$ is compactly generated.
	\end{enumerate}
\end{thm}

\begin{lemma}\label{lem:epis are surjective}
	Let $F\colon\mathcal A\to\mathcal B$ be a homological epimorphism. Then the image of $F$ split-generates $\mathcal B$.
\end{lemma}
\begin{proof}
	Let $M\in\mathrm{Mod}\text{-}\mathcal B$. By \cite[Corollary 1.4.6]{DrinfeldGaitsgory2015}, it suffices to show that $M\cong0$ whenever $M(F(A))$ is acyclic for all $A\in\mathcal A$. This follows from full faithfulness of $F^*$ and the fact that $\mathcal A$ compactly generates $\mathrm{Mod}\text{-}\mathcal A$.
\end{proof}

\begin{conj}\label{conj:epis transfer E_n structure}
	Let $F\colon\mathcal A\to\mathcal B$ be a homological epimorphism of $A_\infty$-algebras. If $\tilde{\mathcal A}$ is an $E_n$-algebra enhancing $A$, then there is an $E_n$-algebra $\tilde{\mathcal B}$ enhancing $\mathcal B$ and a map $\tilde F\colon\tilde{\mathcal A}\to\tilde{\mathcal B}$ enhancing $F$.
\end{conj}
This is the homotopy version of the classical fact that ring epimorphisms with commutative domain have commutative codomain \cite{Silver1967}.

\subsection{$A_n$ sectors}

Returning to symplectic geometry, we consider a class of generalized stabilizations
\[
\bar F\langle n\rangle=\left(\bar F\times\C\right)\setminus\left(\sigma_0\cup\dotsc\cup\sigma_n\right)
\]
associated to a Liouville manifold $\bar F$, where $\sigma_j$ has divisor $\bar F\times\left\{e^{\frac{2\pi j}n}\right\}$. In this case, $\bar F\langle1\rangle\cong\Sigma F$, and in general we think of $\bar F\langle n\rangle$ as describing $A_n$ quiver representations in $\bar F$. To justify this viewpoint, we have

\begin{prop}\label{prop:A_n sector presentation}
	Suppose $\bar F$ satisfies stop removal, and let $n\ge1$. Fix a graded inclusion $i_0\colon\Sigma\bar F\to\bar F\langle n\rangle$ at $\sigma_0$, and take the corresponding graded inclusion $i_j$ at $\sigma_j$ to be the minimal counterclockwise rotation of $i_0$. Then the following hold.
	\begin{enumerate}[a)]
		\item\label{A_n Orlov ff} The resulting Orlov functors $\imath_j$ are fully faithful.
		\item\label{A_n forward stopped} $\hom_{\W(\bar F\langle n\rangle)}(\imath_jL,\imath_{j+k}L')$ is acyclic for $k\ne0,1$, where the indices are cyclically ordered.
		\item\label{A_n iterated cone} For $j=0,\dotsc,n-1$, there are natural transformations
		\[
		R_j\colon \imath_j\to\imath_{j+1}[1]
		\]
		for which $\imath_n[1]$ is isomorphic to the twisted complex
		\[
		\tw\left(\imath_0\xrightarrow{R_0}\imath_1\xrightarrow{R_1}\dotsm\xrightarrow{R_{n-2}}\imath_{n-1}\right).
		\]
	\end{enumerate}
\end{prop}
\begin{proof}
	\eqref{A_n Orlov ff} and \eqref{A_n forward stopped} each follow from Proposition \ref{prop:forward stopped inclusions}, since all the inclusions $i_j$ are forward stopped.
	
	\begin{figure}
		\def\svgwidth{6cm}
		%% Creator: Inkscape inkscape 0.92.4, www.inkscape.org
%% PDF/EPS/PS + LaTeX output extension by Johan Engelen, 2010
%% Accompanies image file 'partial_An.pdf' (pdf, eps, ps)
%%
%% To include the image in your LaTeX document, write
%%   \input{<filename>.pdf_tex}
%%  instead of
%%   \includegraphics{<filename>.pdf}
%% To scale the image, write
%%   \def\svgwidth{<desired width>}
%%   \input{<filename>.pdf_tex}
%%  instead of
%%   \includegraphics[width=<desired width>]{<filename>.pdf}
%%
%% Images with a different path to the parent latex file can
%% be accessed with the `import' package (which may need to be
%% installed) using
%%   \usepackage{import}
%% in the preamble, and then including the image with
%%   \import{<path to file>}{<filename>.pdf_tex}
%% Alternatively, one can specify
%%   \graphicspath{{<path to file>/}}
%% 
%% For more information, please see info/svg-inkscape on CTAN:
%%   http://tug.ctan.org/tex-archive/info/svg-inkscape
%%
\begingroup%
  \makeatletter%
  \providecommand\color[2][]{%
    \errmessage{(Inkscape) Color is used for the text in Inkscape, but the package 'color.sty' is not loaded}%
    \renewcommand\color[2][]{}%
  }%
  \providecommand\transparent[1]{%
    \errmessage{(Inkscape) Transparency is used (non-zero) for the text in Inkscape, but the package 'transparent.sty' is not loaded}%
    \renewcommand\transparent[1]{}%
  }%
  \providecommand\rotatebox[2]{#2}%
  \newcommand*\fsize{\dimexpr\f@size pt\relax}%
  \newcommand*\lineheight[1]{\fontsize{\fsize}{#1\fsize}\selectfont}%
  \ifx\svgwidth\undefined%
    \setlength{\unitlength}{307.25198865bp}%
    \ifx\svgscale\undefined%
      \relax%
    \else%
      \setlength{\unitlength}{\unitlength * \real{\svgscale}}%
    \fi%
  \else%
    \setlength{\unitlength}{\svgwidth}%
  \fi%
  \global\let\svgwidth\undefined%
  \global\let\svgscale\undefined%
  \makeatother%
  \begin{picture}(1,1.04884378)%
    \lineheight{1}%
    \setlength\tabcolsep{0pt}%
    \put(0,0){\includegraphics[width=\unitlength,page=1]{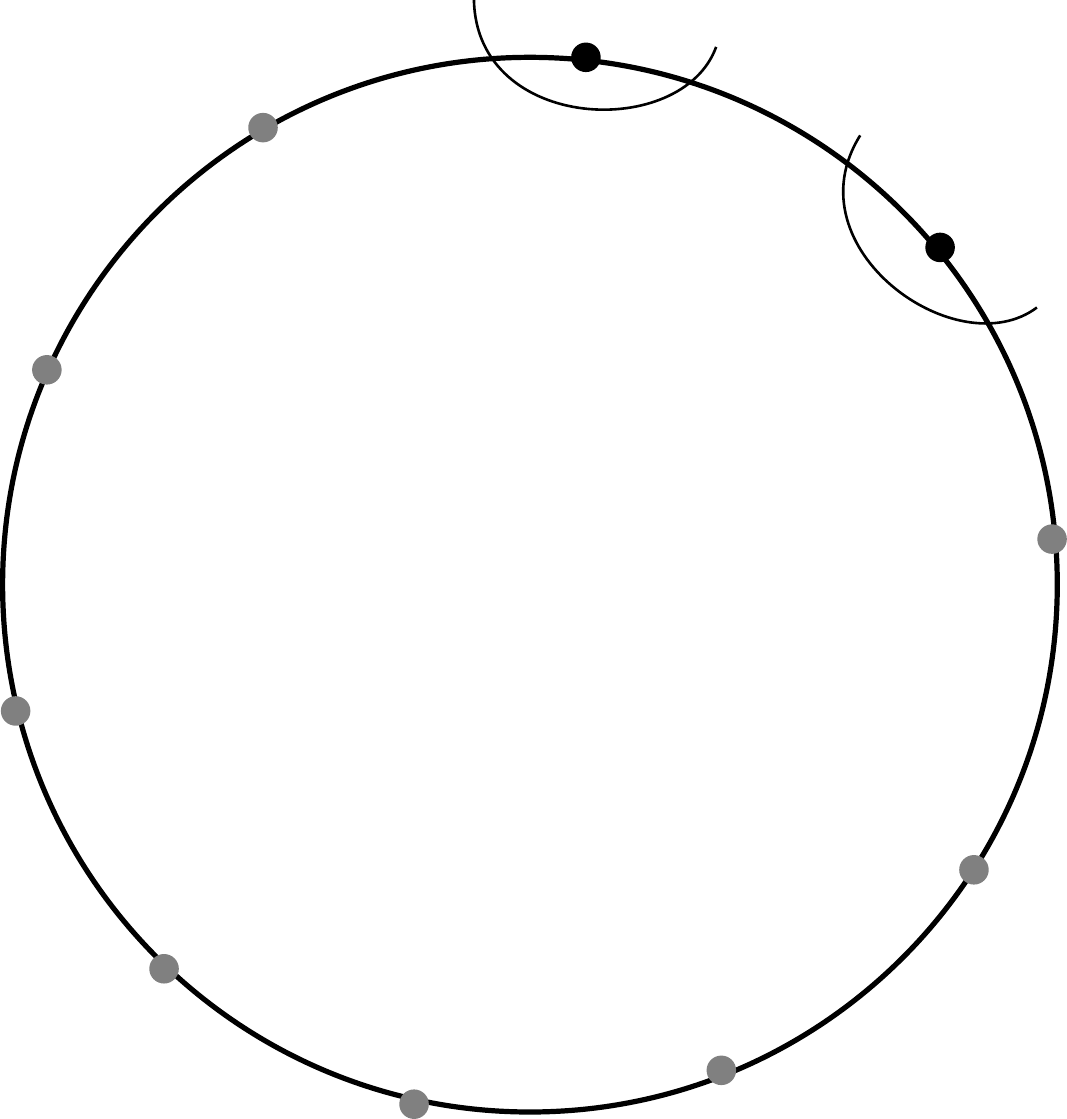}}%
    \put(0.42961475,0.91944729){\color[rgb]{0,0,0}\makebox(0,0)[lt]{\lineheight{1.25}\smash{\begin{tabular}[t]{l}$i_{j+1}$\end{tabular}}}}%
    \put(0.83970164,0.71196285){\color[rgb]{0,0,0}\makebox(0,0)[lt]{\lineheight{1.25}\smash{\begin{tabular}[t]{l}$i_j$\end{tabular}}}}%
  \end{picture}%
\endgroup%

		\caption{$i_j$ and $i_{j+1}$ become isotopic up to shift after removing the extra stops.}\label{fig:partial A_n}
	\end{figure}
	
	For \eqref{A_n iterated cone}, note that $i_j$ is isotopic to $i_{j+1}[1]$ in the larger sector $\Sigma_j\bar F=(\bar F\times\C)\setminus(\sigma_j\cup\sigma_{j+1})$, as illustrated in Figure \ref{fig:partial A_n}. By Proposition \ref{prop:isotopic inclusions}, there is a natural isomorphism
	\[
	R_j^+\colon \imath_j^+\to\imath_{j+1}^+[1],
	\]
	where $\imath_j^+$ and $\imath_{j+1}^+$ are the Orlov functors into $\Sigma_j\bar F$. On the other hand, the tautological inclusion
	\[
	\W(\bar F\langle n\rangle)\into\W(\Sigma_j\bar F)
	\]
	induces an inclusion of chain complexes
	\begin{equation}\label{A_n equality of natural transformations}
	\hom(\imath_j,\imath_{j+1}[1])\into\hom(\imath_j^+,\imath_{j+1}^+[1]),
	\end{equation}
	and for sufficiently nice Hamiltonians \eqref{A_n equality of natural transformations} is an \emph{equality}. $R_j$ is now the lift of $R_j^+$ under this identification.
	
	For the isomorphism, it suffices to prove that the larger twisted complex
	\[
	\imath_T=\tw\left(\imath_0\xrightarrow{R_0}\dotsm\xrightarrow{R_{n-1}}\imath_n\right)
	\]
	is the zero functor. By stop removal, the objects $\imath_jL$ for $j=0,\dotsc,n-1$ split-generate $\W(\bar F\langle n\rangle)$, so it is in fact enough to check that $\hom_{\W(\bar F\langle n\rangle)}(\imath_jL,\imath_TL')$ is acyclic for those $j$ and all $L,L'$. By \eqref{A_n forward stopped},
	\[
	\hom_{\W(\bar F\langle n\rangle)}(\imath_jL,\imath_TL')\cong\hom_{\W(\bar F\langle n\rangle)}\left(\imath_jL,\tw\left(\imath_jL'\xrightarrow{R_j(L')}\imath_{j+1}L'\right)\right).
	\]
	On the other hand, as with \eqref{A_n equality of natural transformations},
	\[
	\hom_{\W(\bar F\langle n\rangle)}\left(\imath_jL,\tw\left(\imath_jL'\xrightarrow{R_j(L')}\imath_{j+1}L'\right)\right)\cong\hom_{\W(\Sigma_j\bar F)}\left(\imath^+_jL,\tw\left(\imath^+_jL'\xrightarrow{R^+_j(L')}\imath^+_{j+1}L'\right)\right),
	\]
	and this last complex is acyclic because $R_j^+$ is an isomorphism.
\end{proof}

\begin{rmk}\label{rmk:A_n formula better proof}
	Proposition \ref{prop:A_n sector presentation} is well known to experts. A proof at the level of objects appears in \cite{GanatraPardonShende2018}, and a complete description of $\W(\bar F\langle n\rangle)$ appears in \cite{Tanaka2019}. A conceptually simpler proof in the same spirit would be to establish the result for $\bar F=\mathrm{pt}$ and obtain the general case by K\"unneth. This would have the additional benefit of not requiring stop removal. However, such a proof would require compatibility of the pushforward and K\"unneth functors, which is beyond the scope of this paper.
\end{rmk}

\subsection{Gluing}\label{sec:gluing}

We now move to the situation where we have
\begin{itemize}
	\item two sectors $M$ and $N$,
	\item chosen stops $\sigma_M$ of $M$ and $\sigma_N$ of $N$ with fibers $\bar F_M$ and $\bar F_N$, respectively, and 
	\item a Liouville isomorphism $\bar F_M\cong\bar F_N$.
\end{itemize}
After possibly modifying $N$ by a trivial inclusion, we may arrange that $\bar F_M=\bar F_N=\bar F$, or in other words that the isomorphism preserves the Liouville form on the nose. This does not change any of the Fukaya categories, and by Proposition \ref{prop:isotopic inclusions} it does not change and of the Orlov functors either.

This done, we can glue $M$ and $N$ along $\bar F\times\C_{\lvert\Re\rvert\le1}\cong\Sigma F$ to obtain the glued sector $M\glu{\sigma_M}{\sigma_N}N$, where $\sigma_M(p,z)$ is identified with $(p,z)$ and $\sigma_N(p,z)$ is identified with $(p,-z)$. When the stops are clear from context, we will shorten the notation to $M\cup N$, not to be confused with the disjoint union $M\amalg N$.

Alternatively, we can be more fancy and form the \emph{stopped gluing}
\[
M\sglu{\sigma_M}{\sigma_N}N:=M\glu{\sigma_M}{\sigma_2}\bar F\langle 3\rangle\glu{\sigma_0}{\sigma_N}N.
\]
Our first order of business is to understand $\W(M\sglu{}{}N)$, or at least the image of the pushforward functors associated to the inclusions of $M$, $N$, and $\bar F\langle3\rangle$.

\begin{figure}
	\def\svgwidth{15.5cm}
	%% Creator: Inkscape inkscape 0.92.4, www.inkscape.org
%% PDF/EPS/PS + LaTeX output extension by Johan Engelen, 2010
%% Accompanies image file 'grothendieck_vs_cocomma.pdf' (pdf, eps, ps)
%%
%% To include the image in your LaTeX document, write
%%   \input{<filename>.pdf_tex}
%%  instead of
%%   \includegraphics{<filename>.pdf}
%% To scale the image, write
%%   \def\svgwidth{<desired width>}
%%   \input{<filename>.pdf_tex}
%%  instead of
%%   \includegraphics[width=<desired width>]{<filename>.pdf}
%%
%% Images with a different path to the parent latex file can
%% be accessed with the `import' package (which may need to be
%% installed) using
%%   \usepackage{import}
%% in the preamble, and then including the image with
%%   \import{<path to file>}{<filename>.pdf_tex}
%% Alternatively, one can specify
%%   \graphicspath{{<path to file>/}}
%% 
%% For more information, please see info/svg-inkscape on CTAN:
%%   http://tug.ctan.org/tex-archive/info/svg-inkscape
%%
\begingroup%
  \makeatletter%
  \providecommand\color[2][]{%
    \errmessage{(Inkscape) Color is used for the text in Inkscape, but the package 'color.sty' is not loaded}%
    \renewcommand\color[2][]{}%
  }%
  \providecommand\transparent[1]{%
    \errmessage{(Inkscape) Transparency is used (non-zero) for the text in Inkscape, but the package 'transparent.sty' is not loaded}%
    \renewcommand\transparent[1]{}%
  }%
  \providecommand\rotatebox[2]{#2}%
  \newcommand*\fsize{\dimexpr\f@size pt\relax}%
  \newcommand*\lineheight[1]{\fontsize{\fsize}{#1\fsize}\selectfont}%
  \ifx\svgwidth\undefined%
    \setlength{\unitlength}{462.83806942bp}%
    \ifx\svgscale\undefined%
      \relax%
    \else%
      \setlength{\unitlength}{\unitlength * \real{\svgscale}}%
    \fi%
  \else%
    \setlength{\unitlength}{\svgwidth}%
  \fi%
  \global\let\svgwidth\undefined%
  \global\let\svgscale\undefined%
  \makeatother%
  \begin{picture}(1,0.25964934)%
    \lineheight{1}%
    \setlength\tabcolsep{0pt}%
    \put(0,0){\includegraphics[width=\unitlength,page=1]{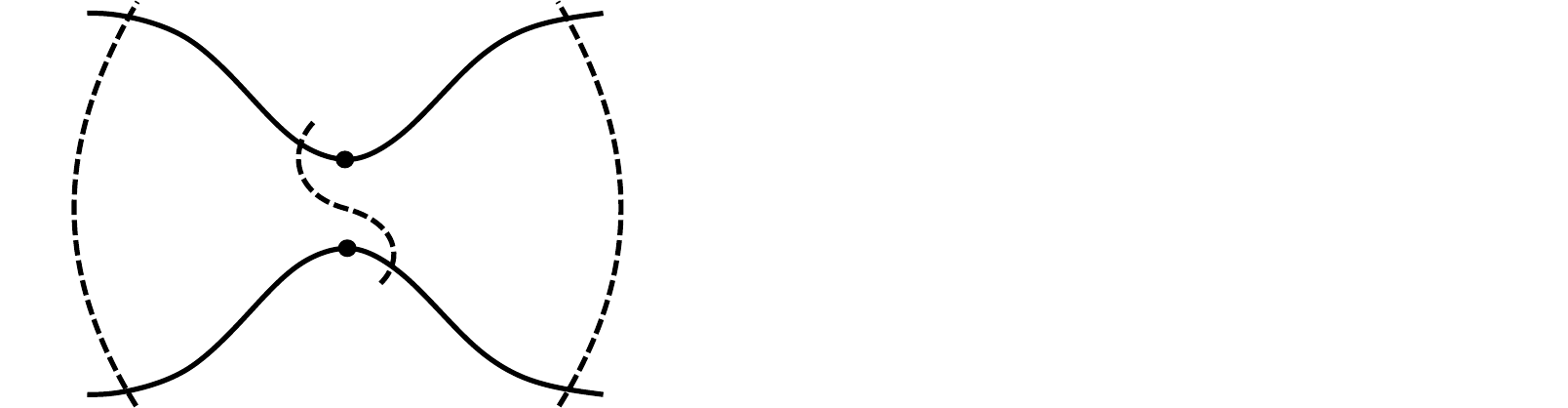}}%
    \put(-0.00143476,0.10739887){\color[rgb]{0,0,0}\makebox(0,0)[lt]{\lineheight{1.25}\smash{\begin{tabular}[t]{l}$N$\end{tabular}}}}%
    \put(0.40488986,0.10294269){\color[rgb]{0,0,0}\makebox(0,0)[lt]{\lineheight{1.25}\smash{\begin{tabular}[t]{l}$M$\end{tabular}}}}%
    \put(0.25337897,0.11198214){\color[rgb]{0,0,0}\makebox(0,0)[lt]{\lineheight{1.25}\smash{\begin{tabular}[t]{l}$\bar F$\end{tabular}}}}%
    \put(0,0){\includegraphics[width=\unitlength,page=2]{grothendieck_vs_cocomma.pdf}}%
    \put(0.55818614,0.1077188){\color[rgb]{0,0,0}\makebox(0,0)[lt]{\lineheight{1.25}\smash{\begin{tabular}[t]{l}$N$\end{tabular}}}}%
    \put(0.96451074,0.10326262){\color[rgb]{0,0,0}\makebox(0,0)[lt]{\lineheight{1.25}\smash{\begin{tabular}[t]{l}$M$\end{tabular}}}}%
    \put(0.81299978,0.13751016){\color[rgb]{0,0,0}\makebox(0,0)[lt]{\lineheight{1.25}\smash{\begin{tabular}[t]{l}$\bar F$\end{tabular}}}}%
    \put(0,0){\includegraphics[width=\unitlength,page=3]{grothendieck_vs_cocomma.pdf}}%
  \end{picture}%
\endgroup%

	\caption{Grothendieck construction on the left, mutated diagram on the right. The arrows represent Reeb flow.}\label{fig:grothendieck vs cocomma}
\end{figure}

\begin{rmk}
	Ganatra, Pardon, and Shende also study the stopped gluing in their proof of the gluing formula \cite{GanatraPardonShende2018}. In their case, they identify $\W(M\sglu{}{}N)$ with the Grothendieck construction of the span
	\[
	\W(M)\xleftarrow{\imath_{\sigma_M}}\W(\bar F)\xrightarrow{\imath_{\sigma_N}}\W(N).
	\]
	In our case, we will use a semiorthogonal presentation associated to the mutated bimodule diagram
	\begin{equation}\label{stopped gluing bimodule diagram}
	\W(M)\xrightarrow{(\imath_{\sigma_M})^\dagger}\W(\bar F)\xrightarrow{\imath_{\sigma_N}}\W(N).
	\end{equation}
	Geometrically, the difference in presentation comes from privileging a different collection of objects in $\W(\bar F\langle3\rangle)$, see Figure \ref{fig:grothendieck vs cocomma}. This leads to formulas which are equivalent but better suited to our purposes.
\end{rmk}

Explicitly, the semiorthogonal gluing of \eqref{stopped gluing bimodule diagram} has the form
\[
\overset{\stoppedarrow}{\mathcal C}=\begin{vmatrix}\W(M)&\Gamma(\imath_{\sigma_M})&\mathbf 0\\
&\W(\bar F)&\Gamma^\dagger(\imath_{\sigma_N})\\
&&\W(N)
\end{vmatrix}.
\]

\begin{prop}\label{prop:stopped gluing formula}
	If $\bar F$ satisfies stop removal, then there is a fully faithful functor $\overset{\stoppedarrow}{\Phi}\colon\overset{\stoppedarrow}{\mathcal C}\to\W(M\sglu{}{}N)$ which agrees with the pushforward functors on $\W(M)$ and $\W(N)$.
\end{prop}
\begin{rmk}
	The only use of stop removal in the proof of Proposition \ref{prop:stopped gluing formula} comes from invoking Proposition \ref{prop:A_n sector presentation}, so by Remark \ref{rmk:A_n formula better proof} it should not be viewed as an essential hypothesis for the stopped gluing formula. On the other hand, our main interest will be in the ``directed gluing'' formula of Corollary \ref{cor:directed gluing formula} below, where it is indeed essential, so for our purposes we lose nothing by requiring it at this stage.
\end{rmk}
\begin{proof}
	To begin, note that essentially everything in sight is forward stopped. To be more precise, the inclusion $i_M$ is forward stopped into $\sigma_3$, $i_N$ is forward stopped into $\sigma_1$, and $i_3$ is forward stopped into $\sigma_1\cup\sigma_3$. In each case, this can be seen by extending the inclusions of $\Sigma\bar F$ associated to the respective Orlov functors to inclusions of $\bar F\times U$, where $U\subset T^*\R$ is a neighborhood of the zero section, see Figure \ref{fig:sample forward stopping}. A wrapping Hamiltonian extending $H_\Sigma$ from Construction \ref{const:orlov via ham} will preserve the positive half $\bar F\times U_{p\ge0}\subset\bar F\times U$, which provides the edge of the stopping submanifold $W$. As a consequence, we can choose wrapping Hamiltonians with no chords starting in $M$ and ending in $N$, and we choose to present $\W(M\sglu{}{}N)$ using such Hamiltonians.
	
	On objects, $\overset{\stoppedarrow}{\Phi}$ is given by
	\[
	\overset{\stoppedarrow}{\Phi}(L)=\begin{cases}(i_M)_*L&L\in\W(M)\\
	(i_{\langle\bar F\rangle})_*\tw\left(\imath_0\xrightarrow{R_1}\imath_1\xrightarrow{R_2}\imath_2\right)L[-1]&L\in\W(\bar F)\\
	(i_N)_*L&L\in\W(N).\end{cases}
	\]
	On diagonal morphisms, it is given by those same functors. On off-diagonal morphisms, it is induced by the map of $(\W(\bar F),\W(M))$-bimodules
	\begin{equation}\label{stopped gluing bimodule map}
	\begin{tikzcd}
	\hspace{-.5cm}\Gamma(\imath_{\sigma_M})=(\imath_{\sigma_M},\mathrm{Id})^*\Delta_{\W(M)}\ar[d,"(i_M)_*"]&\\
	\left((i_M)_*\circ\imath_{\sigma_M},(i_M)_*\right)^*\Delta_{\W(M\sglu{\sigma_M}{\sigma_N}N)}\ar[r,"\text{Prop }\ref{prop:isotopic inclusions}","\cong" below]&\left((i_{\bar F\langle3\rangle})_*\circ\imath_2[-1],(i_M)_*\right)^*\Delta_{\W(M\sglu{\sigma_M}{\sigma_N}N)}\ar[d]\\
	&\left((i_{\bar F\langle3\rangle})_*\circ\tw\left(\imath_0\to\imath_1\to\imath_2\right)[-1],(i_M)_*\right)^*\Delta_{\W(M\sglu{\sigma_M}{\sigma_N}N)}
	\end{tikzcd}
	\end{equation}
	and similarly for $\Gamma^\dagger(\imath_{\sigma_N})$. The oddness of the shift in the isomorphism $(i_M)_*\circ\imath_{\sigma_M}\cong(i_{\bar F\langle3\rangle})_*\circ\imath_2[-1]$ comes from the half-rotation of $\C_{\lvert\Re\rvert<1}$ in the construction of the gluing. As with the construction of the Viterbo transfer map, the precise value is a matter of convention -- other choices give presentations of the same (pretriangulated) category which differ by shifts. It is immediate that $\overset\stoppedarrow\Phi$ satisfies the equations of an $A_\infty$ functor, except possibly when applied to a word of morphisms starting in $\W(M)$ and ending in $\W(N)$. These terms vanish because we have chosen a model of $\W(M\sglu{}{}N)$ with genuinely no morphisms between branes in $M$ and branes in $N$.

	It remains to show that $\overset{\stoppedarrow}{\Phi}$ is fully faithful. By Propositions \ref{prop:forward stopped inclusions} and \ref{prop:A_n sector presentation}\eqref{A_n iterated cone}, $\overset{\stoppedarrow}{\Phi}$ is fully faithful on diagonal morphisms. For off-diagonal morphisms, we again consider only $\Gamma(\imath_{\sigma_M})$. There, Proposition \ref{prop:forward stopped inclusions} implies that the first vertical arrow of \eqref{stopped gluing bimodule map} is a quasi-isomorphism, and for the second vertical arrow we use the fact that $\hom\left((i_M)_*,\imath_0\right)$ and $\hom\left((i_M)_*,\imath_1\right)$ are acyclic, again by forward stoppedness.
	
	\begin{figure}
		\def\svgwidth{10cm}
		%% Creator: Inkscape inkscape 0.92.4, www.inkscape.org
%% PDF/EPS/PS + LaTeX output extension by Johan Engelen, 2010
%% Accompanies image file 'sample_forward_stopping.pdf' (pdf, eps, ps)
%%
%% To include the image in your LaTeX document, write
%%   \input{<filename>.pdf_tex}
%%  instead of
%%   \includegraphics{<filename>.pdf}
%% To scale the image, write
%%   \def\svgwidth{<desired width>}
%%   \input{<filename>.pdf_tex}
%%  instead of
%%   \includegraphics[width=<desired width>]{<filename>.pdf}
%%
%% Images with a different path to the parent latex file can
%% be accessed with the `import' package (which may need to be
%% installed) using
%%   \usepackage{import}
%% in the preamble, and then including the image with
%%   \import{<path to file>}{<filename>.pdf_tex}
%% Alternatively, one can specify
%%   \graphicspath{{<path to file>/}}
%% 
%% For more information, please see info/svg-inkscape on CTAN:
%%   http://tug.ctan.org/tex-archive/info/svg-inkscape
%%
\begingroup%
  \makeatletter%
  \providecommand\color[2][]{%
    \errmessage{(Inkscape) Color is used for the text in Inkscape, but the package 'color.sty' is not loaded}%
    \renewcommand\color[2][]{}%
  }%
  \providecommand\transparent[1]{%
    \errmessage{(Inkscape) Transparency is used (non-zero) for the text in Inkscape, but the package 'transparent.sty' is not loaded}%
    \renewcommand\transparent[1]{}%
  }%
  \providecommand\rotatebox[2]{#2}%
  \newcommand*\fsize{\dimexpr\f@size pt\relax}%
  \newcommand*\lineheight[1]{\fontsize{\fsize}{#1\fsize}\selectfont}%
  \ifx\svgwidth\undefined%
    \setlength{\unitlength}{181.55035798bp}%
    \ifx\svgscale\undefined%
      \relax%
    \else%
      \setlength{\unitlength}{\unitlength * \real{\svgscale}}%
    \fi%
  \else%
    \setlength{\unitlength}{\svgwidth}%
  \fi%
  \global\let\svgwidth\undefined%
  \global\let\svgscale\undefined%
  \makeatother%
  \begin{picture}(1,0.52146504)%
    \lineheight{1}%
    \setlength\tabcolsep{0pt}%
    \put(0,0){\includegraphics[width=\unitlength,page=1]{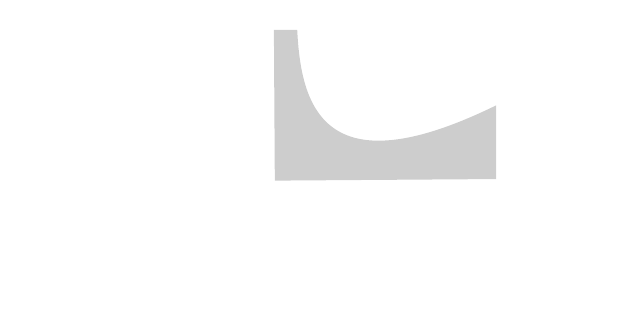}}%
    \put(-0.00365773,0.22112103){\color[rgb]{0,0,0}\makebox(0,0)[lt]{\lineheight{1.25}\smash{\begin{tabular}[t]{l}$N$\end{tabular}}}}%
    \put(0.90952493,0.21268172){\color[rgb]{0,0,0}\makebox(0,0)[lt]{\lineheight{1.25}\smash{\begin{tabular}[t]{l}$M$\end{tabular}}}}%
    \put(0,0){\includegraphics[width=\unitlength,page=2]{sample_forward_stopping.pdf}}%
    \put(0.46075287,0.257678){\color[rgb]{0,0,0}\makebox(0,0)[lt]{\lineheight{1.25}\smash{\begin{tabular}[t]{l}$W$\end{tabular}}}}%
    \put(0.36263957,0.48798603){\color[rgb]{0,0,0}\makebox(0,0)[lt]{\lineheight{1.25}\smash{\begin{tabular}[t]{l}$\sigma_3$\end{tabular}}}}%
  \end{picture}%
\endgroup%

		\caption{The forward stopping manifold for $i_M$.}\label{fig:sample forward stopping}
	\end{figure}
\end{proof}

When $\bar F$ satisfies stop removal, Proposition \ref{prop:stopped gluing formula} immediately gives a formula for the ordinary gluing. Indeed, $M\cup N$ is deformation equivalent to $\left(M\sglu{}{}N\right)\cup\left(\sigma_1\cup\sigma_3\right)$, and it is easy to determine the images of $\overset{\stoppedarrow}{\Phi}^{-1}\imath_1$ and $\overset{\stoppedarrow}{\Phi}^{-1}\imath_3$. In our case, it will be convenient to perform this stop removal in two steps, beginning with $\sigma_3$. The result is the \emph{directed gluing}
\[
M\dglu{\sigma_M}{\sigma_N}N:=M\glu{\sigma_M}{\sigma_2}\bar F\langle2\rangle\glu{\sigma_0}{\sigma_N}N.
\]

\begin{cor}\label{cor:directed gluing formula}
	Set
	\[
	\overset{\rightarrow}{\mathcal C}=\begin{vmatrix}
	\W(M)&\Gamma^\dagger(\imath_{\sigma_N})[1]\otimes_{\W(\bar F)}\Gamma(\imath_{\sigma_M})\\
	&\W(N)
	\end{vmatrix}.
	\]
	If $\bar F$ satisfies stop removal, then there is a fully faithful functor $\overset{\rightarrow}{\Phi}\colon\overset{\rightarrow}{\mathcal C}\to\W(M\dglu{}{}N)$ which agrees with the pushforward functors on $\W(M)$ and $\W(N)$.
\end{cor}
\begin{proof}
	Because $\overset{\stoppedarrow}{\Phi}$ is isomorphic to $\imath_3$ on $\W(\bar F)$, it is enough to show that $\overset{\rightarrow}{\mathcal C}\cong\overset{\stoppedarrow}{\mathcal C}/\W(\bar F)$ compatibly with the inclusions of $\W(M)$ and $\W(N)$. This follows from the observation that $\overset{\stoppedarrow}{\mathcal C}/\W(\bar F)$ has a semiorthogonal presentation $\langle\W(M),\W(N)\rangle$ with gluing bimodule
	\begin{align*}
	\hom(L_M,L_N)&=\bigoplus_{j\ge1}\bigoplus_{B_1,\dotsc,B_j\in\W(\bar F)}\hom_{\overset{\stoppedarrow}{\mathcal C}}(B_j,L_N)\otimes\dotsm\otimes\hom_{\overset{\stoppedarrow}{\mathcal C}}(L_M,B_1)[j]\\
	&=\Gamma^\dagger(\imath_{\sigma_N})[1]\otimes_{\W(\bar F)}\Gamma(\imath_{\sigma_M})(L_N,L_M).
	\end{align*}
	Here, we have used the fact that $\hom_{\overset{\stoppedarrow}{\mathcal C}}(L_M,L_N)=0$.
\end{proof}

To produce a formula for the ordinary gluing $M\cup N$, we need to identify the functor to $\overset{\rightarrow}{\mathcal C}$ which, under $\overset{\rightarrow}{\Phi}$, corresponds to $\imath_1$.

\begin{prop}\label{prop:directed to undirected gluing}
	In the situation of Corollary \ref{cor:directed gluing formula}, assume that all categories and functors are strictly unital. Then, under $\overset{\rightarrow}{\Phi}$, the cone of the natural transformation
	\[
	T\in\hom^0_{\Fun(\W(\bar F),\,\overset{\rightarrow}{\mathcal C})}\left(\imath_{\sigma_M}[1],\imath_{\sigma_N}\right)
	\]
	given for $a_i\in\hom_{\W(\bar F)}(L_{i-1},L_i)$ by
	\begin{equation}\label{cocomma nat transformation}
	T(a_d,\dotsc,a_1)=(-1)^{(\lvert a_1\rvert-1)+\dotsm+(\lvert a_d\rvert-1)}\mathbf 1_{\imath_{\sigma_N}(L_d)}\otimes a_d\otimes\dotsm\otimes a_1\otimes\mathbf 1_{\imath_{\sigma_M}(L_0)}
	\end{equation}
	is $(i_{\bar F\langle2\rangle})_*\circ\imath_1$. Here $\mathbf1_{\imath_{\sigma_M}(L_0)}$ is the degree $1$ morphism in $\hom(\imath_{\sigma_M}(L_0)[1],\imath_{\sigma_M}(L_0))$ coming from the unit of $\imath_{\sigma_M}(L_0)$.
\end{prop}
\begin{proof}
	That $T$ is in fact a natural transformation is a calculation. To see that its cone is $\imath_1$, we'll begin by lifting it to $\overset\stoppedarrow{\mathcal C}$. To that end, decompose the projection functor $\overset\stoppedarrow{\mathcal C}\to\overset\rightarrow{\mathcal C}$ as a composition
	\begin{equation}\label{algebraic stop removal}\begin{tikzcd}
	\overset\stoppedarrow{\mathcal C}\ar{r}{Q}&\overset\stoppedarrow{\mathcal C}/\W(\bar F)\ar{r}{S}&\overset\rightarrow{\mathcal C}\end{tikzcd},
	\end{equation}
	where $Q$ is the quotient functor and $S$ sends the objects of $\W(\bar F)$ to zero. More explicitly, $S$ is the functor which restricts to the identity functor of $\W(M)$ and $W(N)$, and which on sequences of composable morphisms $w_i\in\hom(L_i,L_{i+1})$ is given by
	\[
	S^d(w_d,\dotsc,w_0)=\begin{cases}w_d\otimes\dotsm\otimes w_0&L_0,L_d\not\in\W(\bar F)\mbox{ and } L_1,\dotsm,L_{d-1}\in\W(\bar F)\\
	0&\mbox{otherwise}\end{cases}
	\]
	Here we have written $w$ for morphisms in the quotient category to emphasize that they are presented by words, and $S$ concatenates them. Now $\overset\rightarrow{\mathcal C}$ is \emph{isomorphic} at chain level to a full subcategory of $\overset\stoppedarrow{\mathcal C}/\W(\bar F)$, and $S$ is a quasi-inverse of that inclusion, so we may view $S$ as an autoequivalence of $\overset\stoppedarrow{\mathcal C}/\W(\bar F)$ which is isomorphic to the identity autoequivalence.
	
	We can now express $T$ as the whiskering $(S\circ Q)^2(\mathbf e_N^\dagger,\mathbf e_M)$, where $\mathbf e_M\in\hom^1(\imath_{\sigma_M}[1],\mathrm{Id}_{\W(\bar F)})$ has leading term
	\[
	\mathbf e_M^0(L)=\mathbf 1_{\imath_{\sigma_M}(L)}\in\Gamma(\imath_{\sigma_M})(L,\imath_{\sigma_M}(L))=\mathrm{end}(\imath_{\sigma_M}(L))
	\]
	and no higher terms, and similarly with $\mathbf e_N\in\hom^0(\mathrm{Id}_{\W(\bar F)},\imath_{\sigma_N})$. This implies that $\cone(T)=(S\circ Q)(A)$, where $A$ is the twisted complex
	\[
	\tw\left(\imath_{\sigma_M}[2]\xrightarrow{\mathbf e_M}\mathrm{Id}_{\W(\bar F)}[1]\xrightarrow{\mathbf e_N[-1]}\imath_{\sigma_N}\right).
	\]
	Because $S$ is isomorphic to the identity and the $\Phi$ embeddings intertwine $Q$ with stop removal, it suffices to show that $\overset{\stoppedarrow}{\Phi}(A)\cong(i_{\bar F\langle3\rangle})_*\circ\imath_1$. For this, we compute that
	\begin{equation}\label{final Orlov functor long formula}
	\overset{\stoppedarrow}{\Phi}(A)=\tw\biggl((i_M)_*\circ\imath_{\sigma_M}[2]
	\rightarrow(i_{\bar F\langle3\rangle})_*\circ\tw\Bigl(\imath_2\leftarrow\imath_1\leftarrow\imath_0\Bigr)
	\rightarrow(i_N)_*\circ\imath_{\sigma_N}\biggr).
	\end{equation}
	Examining \eqref{stopped gluing bimodule map}, which defines the off-diagonal part of $\overset{\stoppedarrow}{\Phi}$, we see that we can reparenthesize \eqref{final Orlov functor long formula} as
	\[
	\tw\biggl(\tw\Bigl((i_M)_*\circ\imath_{\sigma_M}[2]
	\xrightarrow{\cong}(i_{\bar F\langle3\rangle})_*\circ\imath_2\Bigr)\leftarrow(i_{\bar F\langle3\rangle})_*\circ\imath_1\leftarrow\tw\Bigl((i_{\bar F\langle3\rangle})_*\circ\imath_0 \xrightarrow{\cong}(i_N)_*\circ\imath_{\sigma_N}\Bigr)\biggr)\cong(i_{\bar F\langle3\rangle})_*\circ\imath_1.
	\]
\end{proof}

\subsection{Viterbo bimodules}
Returning to the study of the Viterbo transfer map associated to $M^\mathrm{in}\subset\bar M$, we will be interested in three new sectors built from $V_{M^\mathrm{in}}$. The first, which we call the ``doubled Viterbo sector'', is the directed gluing
\[
V_{M^\mathrm{in}}^{(2)}:=V_{M^\mathrm{in}}\dglu{\sigma_0}{\sigma_0}V_{M^\mathrm{in}}.
\]
We will refer to the first copy of $V_{M^\mathrm{in}}$ as $V_{M^\mathrm{in}}^-$ and the second as $V_{M^\mathrm{in}}^+$, and similarly with their objects and Orlov functors. Let us study its Floer theory via the gluing formulas of Section \ref{sec:gluing}. Assuming $M^{\mathrm{in}}$ satisfies stop removal, the images of the two Orlov functors $\imath_0^\pm$ with domain $\W(\bar M^\mathrm{in})$ split-generate $\W\left(V_{M^\mathrm{in}}^{(2)}\right)$. Corollary \ref{cor:directed gluing formula} then gives us
\begin{equation}\label{doubled viterbo sector formula}
\begin{aligned}
\Perf\,\W\left(V_{M^\mathrm{in}}^{(2)}\right)&=\Perf\,\langle\W(V_{M^\mathrm{in}}^-),\W(V_{M^\mathrm{in}}^+)\rangle\\
&\cong\Perf\,\langle\W(M^\mathrm{in,-}),\W(M^\mathrm{in,+})\rangle,
\end{aligned}
\end{equation}
where the identification $\Perf\,\W(\bar M^\mathrm{in})\xrightarrow{\cong}\Perf\,\W(V_{M^\mathrm{in}})$ is given by $\imath_1[1]$. In these coordinates, $\imath_0^\pm$ are given by the corresponding Viterbo functors, which means the gluing bimodule is
\[
\Gamma^\dagger(\mathcal V^+)[1]\otimes_{\W(\bar M)}\Gamma(\mathcal V^-).
\]
After reparametrizing to remove the shift, this amounts to a quasi-isomorphism of $\W(\bar M^\mathrm{in})$-bimodules
\[
\Gamma^\dagger(\mathcal V)\otimes_{\W(\bar M)}\Gamma(\mathcal V)\cong(\imath_1^+,\imath_1^-[1])^*\Delta_{\W\left(V_{M^\mathrm{in}}^{(2)}\right)}.
\]

To prove Theorem \ref{thm:Viterbo is epic}, we will want geometrically interpret the tautological map
\begin{equation}\label{V in the middle}
\Gamma^\dagger(\mathcal V)\otimes_{\W(\bar M)}\Gamma(\mathcal V)\to\Delta_{\W(\bar M^\mathrm{in})}\otimes_{\W(\bar M^\mathrm{in})}\Delta_{\W(\bar M^\mathrm{in})}\cong\Delta_{\W(\bar M^\mathrm{in})}.
\end{equation}
To this end, we imitate the proof of Proposition \ref{prop:directed to undirected gluing} and again lift to the stopped gluing. In the notation of \eqref{algebraic stop removal}, we obtain a commutative diagram
\begin{equation}\label{lift to stopped gluing}
\begin{tikzcd}
\hspace{-.8cm}\stackrel{\stoppedarrow}{\mathcal C}=\begin{vmatrix}\W(\bar M^{\mathrm{in},-})&\Gamma(\mathcal V^-)&\mathbf 0\\&\W(\bar M)&\Gamma^\dagger(\mathcal V^+)[-1]\\&&\W(\bar M^{\mathrm{in},+})\end{vmatrix}\ar[r,"S\circ Q"]\ar[d,"{A=(\mathrm{Id},\mathcal V,\mathrm{Id})}"]&
\begin{vmatrix}\W(\bar M^{\mathrm{in},-})&\Gamma^\dagger(\mathcal V^+)\otimes_{\W(\bar M)}\Gamma(\mathcal V^-)\\&\W(\bar M^{\mathrm{in},+})\end{vmatrix}\underset{\Perf}{\cong}\W\left(V_{M^\mathrm{in}}^{(2)}\right)\hspace{-2.5cm}\ar[d,"\eqref{V in the middle}"]\\
\begin{vmatrix}\W(\bar M^{\mathrm{in},-})&\Delta&\mathbf 0\\&\W(\bar M^\mathrm{in})&\Delta[-1]\\&&\W(\bar M^{\mathrm{in},+})\end{vmatrix}\ar[r,"S\circ Q"]&\begin{vmatrix}\W(\bar M^{\mathrm{in},-})&\Delta\\&\W(\bar M^{\mathrm{in},+})\end{vmatrix}
\end{tikzcd}
\end{equation}
(where the copies of $\W(\bar M^{\mathrm{in},+})$ have been shifted up relative to the decomposition in \eqref{doubled viterbo sector formula}). Note that $\stackrel{\stoppedarrow}{\mathcal C}$ presents the wrapped Fukaya category of the stopped gluing $V_{M^\mathrm{in}}\sglu{\sigma_0}{\sigma_0}V_{M^\mathrm{in}}$. Now our geometric interpretation of \eqref{V in the middle} will come from a geometric interpretation of the functor $A$.

\begin{prop}\label{prop:geometric interpretation of V in the middle}
	Assume both $\bar M$ and (the completion) $\bar M^\mathrm{in}$ satisfy stop removal. Then ``tripled Viterbo sector''
	\[
	V_{M^\mathrm{in}}^{(3)}:=\left(V_{M^\mathrm{in}}^-\sglu{\sigma_0}{\sigma_0}V_{M^\mathrm{in}}^+\right)\glu{\sigma_3(\bar M\langle3\rangle)}{\sigma_0}V_{M^\mathrm{in}}
	\]
	has wrapped Fukaya category
	\begin{equation}\label{tripled Viterbo sector formula}
	\Perf\,\W\left(V_{M^\mathrm{in}}^{(3)}\right)\cong\Perf\,\begin{vmatrix}\W(\bar M^{\mathrm{in},-})&\Delta&\mathbf 0\\&\W(\bar M^\mathrm{in})&\Delta[-1]\\&&\W(\bar M^{\mathrm{in},+})\end{vmatrix}.
	\end{equation}
	In these coordinates, there are automorphisms $\phi^-$ and $\phi^+$ of $\Gamma(\mathcal V^-)$ and $\Gamma^\dagger(\mathcal V^+)$, respectively, assembling into an autoequivalence $\phi$ of $\overset{\stoppedarrow}{\mathcal C}$ for which $A\cong(i_*)\circ\phi$. Here, $i_*$ is the pushforward functor associated to the inclusion
	\[
	i\colon\left(V_{M^\mathrm{in}}\sglu{\sigma_0}{\sigma_0}V_{M^\mathrm{in}}\right)\into V_{M^\mathrm{in}}^{(3)}.
	\]
\end{prop}
\begin{rmk}
	The autoequivalence $\phi$ should not be seen as essential. Rather, it an artifact of the fact that we are only working with functors up to isomorphism instead of up to canonical isomorphism, c.f. Remark \ref{rmk:canonical homotopies}.
\end{rmk}
\begin{proof}
	By stop removal, the objects in the images of the three Orlov functors $\imath_1$ on each Viterbo sector split-generate $\W\left(V_{M^\mathrm{in}}^{(3)}\right)$. Moreover, $V_{M^\mathrm{in}}^{(3)}$ is deformation equivalent to the gluing
	\[
	V_{M^\mathrm{in}}\glu{\sigma_0}{\sigma_1}\bar M^\mathrm{in}\langle3\rangle,
	\]
	since they come from deformation equivalent Liouville pairs. On the other hand, the inclusion $\bar M^\mathrm{in}\langle3\rangle\into V_{M^\mathrm{in}}^{(3)}$ is forward stopped, so by Proposition \ref{prop:forward stopped inclusions} the corresponding pushforward functor is fully faithful. The formula \eqref{tripled Viterbo sector formula} then follows from Proposition \ref{prop:stopped gluing formula} with $M,N=\Sigma\bar M^\mathrm{in}$.
	
	To identify $A$ with the pushforward $i_*$, let us work componentwise. The subcategories $\W(\bar M^{\mathrm{in},-})$ and $\W(\bar M^{\mathrm{in},+})$ are parametrized by the fully faithful functors $\imath_1^-[1]$ and $\imath_1^+$, respectively, so we obtain the identification from functoriality of the pushforward maps. The middle subcategories $\W(\bar M)$ and $\W(\bar M^\mathrm{in})$ are parametrized by the fully faithful functors $\imath_3$ on $\bar M\langle3\rangle$ and $\imath_1$ on $V_{M^\mathrm{in}}$, respectively. As functors to $\W(V_{M^\mathrm{in}}^{(3)})$, we have isomorphisms
	\[
	\left(i_{\bar M\langle3\rangle}\right)_*\circ\imath_3\cong\left(i_{V_{M^\mathrm{in}}}\right)_*\circ\imath_0[-1]\cong\left(i_{V_{M^\mathrm{in}}}\right)_*\circ\imath_1\circ\mathcal V,
	\]
	which gives the desired identification on the middle subcategories. In particular, it follows that $i_*$ and $A$ agree on generating objects.
	
	It remains to construct the automorphisms $\phi^\pm$ which identify the off-diagonal parts of the functors, and as usual we restrict to the first case. Inverting \eqref{tripled Viterbo sector formula}, we study the corresponding two functors
	\begin{equation}\label{tripled sector, first component of functors into}
	\langle\W(M^{\mathrm{in,-}}),\W(\bar M)\rangle\to\W\left(V_{M^\mathrm{in}}^{(3)}\right).
	\end{equation}
	Looking back at the original parametrization of $\Gamma(\mathcal V)$ in \eqref{stopped gluing bimodule map}, we see that $i_*$ is given by
	\begin{equation}\label{restricted i formula}
	\begin{aligned}
	(i_*)|_{\W(\bar M^{\mathrm{in},-})}&=\imath_1^-[1]\\
	(i_*)|_{\W(\bar M)}&=(i_{\bar M\langle3\rangle})_*\circ\tw(\imath_0\to\imath_1\to\imath_2)[-1]\\
	&\qquad\cong\tw\left(\imath_1^+[1]\circ\mathcal V\to\bigl((i_{\bar M\langle3\rangle})_*\circ\imath_1[-1]\bigr)\to\imath_1^-[1]\circ\mathcal V\right)\\
	(i_*)|_{\Gamma(\mathcal V)}&\colon\Gamma(\mathcal V)\xrightarrow{\cong}\left(\imath_1^-[1]\circ\mathcal V,\imath_1^-[1]\right)^*\Delta_{\W\left(V_{M^\mathrm{in}}^{(3)}\right)}.
	\end{aligned}
	\end{equation}
	On the other hand, $A$ is given by
	\begin{equation}\label{restricted A formula}
	\begin{aligned}
	A|_{\W(\bar M^{\mathrm{in},-})}&=\imath_1^-[1]\\
	A|_{\W(\bar M)}&=(i_{\bar M^\mathrm{in}\langle3\rangle})_*\circ\tw(\imath_0\to\imath_1\to\imath_2)[-1]\circ\mathcal V\\
	&\qquad\cong\tw\left(\imath_1^+[1]\to\bigl((i_{\bar M^\mathrm{in}\langle3\rangle})_*\circ\imath_1[-1]\bigr)\to\imath_1^-[1]\right)\circ\mathcal V\\
	A|_{\Gamma(\mathcal V)}&\colon\Gamma(\mathcal V)\xrightarrow{\cong}\left(\imath_1^-[1]\circ\mathcal V,\imath_1^-[1]\right)^*\Delta_{\W\left(V_{M^\mathrm{in}}^{(3)}\right)}.
	\end{aligned}
	\end{equation}
	It is tempting to stop and declare victory, since this expresses the off-diagonal terms of $i_*$ and $A$ by the same formula, except it is not obvious that the two copies of $\imath_1^-[1]\circ\mathcal V$ correspond to the same piece of $(i_*)|_{\W(\bar M)}\cong A|_{\W(\bar M)}$.
	
	Instead, remove the stops $\sigma_1^+$ and $\sigma_1(\bar M\langle3\rangle)$ to obtain a Liouville sector
	\[
	N\cong V_{M^\mathrm{in}}\glu{\sigma_0}{\sigma_0}V_{M^\mathrm{in}}.
	\]
	Now the restriction of the stop removal map $\mathcal{SR}\colon\W\left(V_{M^\mathrm{in}}^{(3)}\right)\to\W(N)$ to the \emph{nonzero} morphism spaces in the image of \eqref{tripled sector, first component of functors into} is fully faithful because those terms factor through the remaining Orlov functors, and those Orlov functors remain forward stopped. Because $i_*$ and $A$ agree on objects we can lift any natural transformation between $\mathcal{SR}\circ i_*$ and $\mathcal{SR}\circ A$ to one between $i_*$ and $A$. For this, pass to the sectorial Viterbo subdomain $\Sigma\bar M^\mathrm{in}\subset N$ via the sectorial Abouzaid--Seidel restriction map $\mathcal V_\mathrm{AS}^\mathrm{sect}$ of Section \ref{sec:abouzaid--seidel}. $\mathcal V_\mathrm{AS}^\mathrm{sect}$ is itself a quasi-equivalence, again by stop removal and forward stoppedness, so in fact we can lift any natural transformation between $\mathcal V_\mathrm{AS}^\mathrm{sect}\circ\mathcal{SR}\circ i_*$ and $\mathcal V_\mathrm{AS}^\mathrm{sect}\circ\mathcal{SR}\circ A$ to one between $i_*$ and $A_*$.
	Looking back at our formulas \eqref{restricted i formula} and \eqref{restricted A formula}, we see these compositions are canonically isomorphic. Indeed, they are given by
	\begin{align*}
	\Sigma_*&\text{ on }\W(\bar M^\mathrm{in}),\\
	\Sigma_*\circ\mathcal V&\text{ on }\W(\bar M),\text{ and}\\
	\Gamma(V)&\xrightarrow{\cong}(\Sigma_*\circ\mathcal V,\Sigma_*)^*\Delta_{\W(\Sigma\bar M^\mathrm{in})}\text{ on off-diagonal morphisms.}
	\end{align*}
	We thus obtain a canonical isomorphism
	\[
	(i_*)|_{\langle\W(M^{\mathrm{in,-}}),\W(\bar M)\rangle}\cong A|_{\langle\W(M^{\mathrm{in,-}}),\W(\bar M)\rangle}.
	\]
	
	This isomorphism agrees with the previously constructed one on $\W(M^\mathrm{in,-})$, but not obviously on $\W(\bar M)$ -- this is the same issue which prevented us from winning earlier. However, the previously constructed isomorphism $(i_*)|_{\W(\bar M)}\cong A|_{\W(\bar M)}$ descends to an automorphism of $\Sigma_*\circ\mathcal V$, which we can reinterpret as the desired automorphism of $\phi^-$ of $\Gamma(\mathcal V)$.
\end{proof}

\begin{proof}[Proof of Theorem \ref{thm:Viterbo is epic}]
	Recall the diagram \eqref{lift to stopped gluing}, reproduced here in condensed form with an additional piece from Proposition \ref{prop:geometric interpretation of V in the middle}.
	\begin{equation*}%\label{diagram for epimorphism}
	\begin{tikzcd}
	&\bigl\langle\W(\bar M^{\mathrm{in},-}),\W(\bar M),\W(\bar M^{\mathrm{in},+})\bigr\rangle\ar[r,"S\circ Q"]\ar[d,"{A=(\mathrm{Id},\mathcal V,\mathrm{Id})}"]\ar[dl,"\phi" above]&
	\bigl\langle\W(\bar M^{\mathrm{in},-}),\W(\bar M^{\mathrm{in},+})\bigr\rangle\ar[d,"\eqref{V in the middle}"]\\
	\bigl\langle\W(\bar M^{\mathrm{in},-}),\W(\bar M),\W(\bar M^{\mathrm{in},+})\bigr\rangle\ar[r,"i_*"]
	&\bigl\langle\W(\bar M^{\mathrm{in},-}),W(\bar M^\mathrm{in}),\W(\bar M^{\mathrm{in},+})\bigr\rangle\ar[r,"S\circ Q"]&\bigl\langle\W(\bar M^{\mathrm{in},-}),\W(\bar M^{\mathrm{in},+})\bigr\rangle
	\end{tikzcd}
	\end{equation*}
	
	To begin, note that under the equivalence \eqref{tripled Viterbo sector formula}, the bottom arrow $S\circ Q$ corresponds to stop removal for $\sigma_1(V_{M^\mathrm{in}})$. This means the bottom composition $S\circ Q\circ i_*$ is the pushforward associated to the chain
	\begin{equation}\label{lower inclusion chain}
	\left(V_{M^\mathrm{in}}\sglu{\sigma_0}{\sigma_0}V_{M^\mathrm{in}}\right)\into V_{M^\mathrm{in}}^{(3)}\into V_{M^\mathrm{in}}^{(3)}\cup\sigma_1(V_{M^\mathrm{in}}).
	\end{equation}
	On the other hand, $V_{M^\mathrm{in}}^{(3)}\cup\sigma_1(V_{M^\mathrm{in}})$ is deformation equivalent to $V_{M^\mathrm{in}}^{(2)}$ because they correspond to the same Liouville pair (this is the key geometric ingredient), whence it follows that $\eqref{lower inclusion chain}$ is isotopic to the inclusion
	\[
	V_{M^\mathrm{in}}^{(3)}\into V_{M^\mathrm{in}}^{(2)}.
	\]
	By stop removal, we conclude that $S\circ Q\circ i_*$ is a quotient by the full subcategory $\W(\bar M)$. Because $\phi$ is a quasi-equivalence and preserves the semiorthogonal decomposition, the full bottom left composition
	\[
	S\circ Q\circ i_*\circ\phi
	\]
	and hence the the top right composition
	\[
	\eqref{V in the middle}\circ S\circ Q
	\]
	are also quotients by the full subcategory $\W(\bar M)$. On the other hand, the same is true for the upper arrow $S\circ Q$ by construction, so \eqref{V in the middle} must be a quasi-equivalence.
\end{proof}

\section{Spherical Orlov functors}\label{ch:spherical}

\subsection{Spherical swaps}\label{sec:spherical swaps}
Recall from \cite{AnnoLogvinenko2017} that a functor $F\colon\mathcal A\to\mathcal B$ is called \emph{spherical} if it has left and right adjoints $L$ and $R$, and such that
\begin{enumerate}
	\item the twist $\mathbf W=\cone\left(FR\xrightarrow{\mathrm{counit}}\mathrm{Id}_{\mathcal B}\right)$ and dual twist $\mathbf W'=\cone\left(\mathrm{Id}_{\mathcal B}\xrightarrow{\mathrm{unit}}FL\right)[-1]$ are inverse quasi-equivalences, and
	\item the cotwist $\mathbf M=\cone\left(\mathrm{Id}_{\mathcal A}\xrightarrow{\mathrm{unit}}RF\right)[-1]$ and dual cotwist $\mathbf M'=\cone\left(LF\xrightarrow{\mathrm{counit}}\mathrm{Id}_{\mathcal A}\right)$ are inverse quasi-equivalences.
\end{enumerate}
The notion of a spherical functor has a well known reinterpretation coming from GIT.

\begin{thm}[\cite{HalpernLeistnerShipman2016}]\label{thm:HL-S}
	$F$ is spherical if and only if the semiorthogonal gluing $\mathcal C=\begin{vmatrix}\mathcal A&\Gamma^\dagger(F)\\&\mathcal B\end{vmatrix}$ fits into a $4$-periodic sequence of semiorthogonal decompositions. In other words, if and only if $\mathcal A^{\perp\perp\perp\perp}=\mathcal A$ as full pretriangulated subcategories of $\mathcal C$.
	
	In this case, the dual twist $\mathbf W'$ is given by the iterated mutation $\mathbb R_{\mathcal A}\circ\mathbb R_{^{\perp\perp}\mathcal A}$, and the dual cotwist $\mathbf M'$ is given by the iterated mutation $\mathbb L_{\mathcal B}\circ\mathbb L_{\mathcal B^{\perp\perp}}$.
\end{thm}

Recall that the left mutation $\mathbb L_{\mathcal B}\colon\mathcal A^-\to\mathcal A^+$ associated to a pair of semiorthogonal decompositions
\begin{equation}\label{admissible subcategory}
\mathcal C=\langle\mathcal B,\mathcal A^-\rangle=\langle\mathcal A^+,\mathcal B\rangle
\end{equation}
is the $\mathcal A^+$-component of the inclusion $i_{\mathcal A^-}\colon\mathcal A^-\into\mathcal C$ under the second semiorthogonal decomposition. Equivalently, it is the unique functor admitting a degree zero natural transformation
\[
i_{\mathcal A^-}\to i_{\mathcal A^+}\circ\mathbb L_{\mathcal B}
\]
whose cone lands in $\mathcal B$. The story for the right mutations is similar, and it is a theorem of Bondal that left and right mutations are inverse to one another \cite{Bondal1990}. A full triangulated subcategory $\mathcal B\subset\mathcal C$ is called \emph{admissible} precisely when it fits into semiorthogonal decompositions \ref{admissible subcategory}.

Note that components of general semiorthogonal gluings need not be admissible, but in our situation we are guaranteed at least one admissible subcategory.

\begin{lemma}[\cite{Efimov2013}]\label{lem:functor target is admissible}
	Let $F\colon\mathcal A\to\mathcal B$ be any functor between pretriangulated $A_\infty$-categories. Let $\langle\mathcal B,\mathcal A^-\rangle$ and $\langle\mathcal A^+,\mathcal B\rangle$ be the gluings along $\Gamma(F)$ and $\Gamma^\dagger(F)$, respectively. Then there is a canonical (up to homotopy) identification
	\[
	\langle\mathcal B,\mathcal A^-\rangle\cong\langle\mathcal A^+,\mathcal B\rangle
	\]
	which restricts to the identity on $\mathcal B$.
\end{lemma}
\begin{proof}
	Assume without loss of generality that everything is strictly unital. The identity endomorphism of $F$ in $\Fun(\mathcal A,\mathcal B)$ induces a natural transformation
	\[
	T\in\hom^0_{\Fun(\mathcal A,\langle\mathcal B,\mathcal A^-\rangle)}(i_{\mathcal B}\circ F,i_{\mathcal A^-}).
	\]
	One can readily verify that $\cone(T)[-1]$ is fully faithful, that its image is right-orthogonal to $\mathcal B$, that its image and $\mathcal B$ together generate the glued category $\langle\mathcal B,\mathcal A^-\rangle$, and that
	\[
	(i_{\mathcal B},\cone(T)[-1])^*\Delta_{\langle\mathcal B,\mathcal A^-\rangle}\cong\Gamma^\dagger(F).
	\]
\end{proof}

We will not need it, but it is an easy exercise that all admissible subcategories arise in this way. We are now ready to introduce our main algebraic notion.

\begin{defn}\label{def:spherical swap}
	With notation as in Lemma \ref{lem:functor target is admissible}, let
	\[
	\mathcal C=\langle\mathcal B,\mathcal A^-\rangle=\langle\mathcal A^+,\mathcal B\rangle.
	\]
	Define a \emph{spherical swap} of $F$ to be an autoequivalence $S$ of $\mathcal C$ exchanging the full subcategories $\mathcal A^-$ and $\mathcal A^+$.
	
	$S$ is called \emph{$\mathcal A$-positive} if $S\circ i_{\mathcal A^-}\cong i_{\mathcal A^+}$ and \emph{$\mathcal B$-positive} if the composition
	\[
	\mathcal B\cong\mathcal C/\mathcal A^+\xrightarrow{S}\mathcal C/\mathcal A^-\cong\mathcal B
	\]
	is isomorphic to $\mathrm{Id}_{\mathcal B}$. $S$ is \emph{positive} if it is both $\mathcal A$-positive and $\mathcal B$-positive.
\end{defn}

Note that if $S$ is any spherical swap of $F$, then the iterate $S^2$ fixes $\mathcal B$, because it fixes $\mathcal A^+$ and preserves semiorthogonal complements.

What follows is the main result of this section.

\begin{prop}\label{prop:spherical swap iff spherical}
	With notation as in Lemma \ref{lem:functor target is admissible}, the following are equivalent.
	\begin{enumerate}
		\item\label{item:F spherical} $F$ is spherical.
		\item\label{item:pos spherical swap} $F$ admits a positive spherical swap.
		\item\label{item:spherical swap} $F$ admits a spherical swap.
	\end{enumerate}
	
	Moreover, if $S$ is a $\mathcal B$-positive spherical swap of $F$, then the dual twist $\mathbf W'$ is isomorphic to the restriction $S^2|_{\mathcal B}$. Similarly, if $S$ is $\mathcal A$-positive, then the dual cotwist $\mathbf M'$ is isomorphic to the restriction $S^2[2]|_{\mathcal A^+}$.
\end{prop}
\begin{proof}
	\emph{\ref{item:F spherical}$\Rightarrow$\ref{item:pos spherical swap}:} This can be extracted from the proof of \cite[Theorem 3.15]{HalpernLeistnerShipman2016}, which is part of Theorem \ref{thm:HL-S}, but we will reproduce the relevant portion in our notation. We will still allow ourselves to use Theorem \ref{thm:HL-S} in order to minimize computation.
	
	In the coordinates $\mathcal C=\langle\mathcal B,\mathcal A^-\rangle$, we wish to construct an autoequivalence $S$ of $\mathcal C$ such that $S|_{\mathcal A^-}\cong\cone(T)[-1]$, where $T$ is the natural transformation from the proof of Lemma \ref{lem:functor target is admissible}. Such an autoequivalence is automatically a swap because it preserves double orthogonals and $(\mathcal A^+)^{\perp\perp}=\mathcal A^-$ by Theorem \ref{thm:HL-S}. It is also automatically $\mathcal A$-positive.
	
	Define $S|_{\mathcal B}=\cone(U)[-1]$, where
	\[
	U\in\hom^0_{\Fun(\mathcal B,\langle\mathcal B,\mathcal A^-\rangle)}(i_{\mathcal B},i_{\mathcal A^-}\circ L)\cong\hom^0_{\Fun(\mathcal B,\langle\mathcal B,\mathcal A^-\rangle)}(i_{\mathcal B},i_{\mathcal B}\circ F\circ L)
	\]
	is induced by the unit of the adjunction $L\dashv F$. This guarantees that $S$ will be $\mathcal B$-positive. Now, one readily computes via the adjunction that the image of $S|_{\mathcal B}$ lands in $^\perp(\mathcal A^-)=(\mathcal A^+)^\perp$, and this is surjective because together with $\mathcal A^-$ it generates $\mathcal C$. It remains to produce any isomorphism
	\[
	(S\circ i_{\mathcal A^-},S\circ i_{\mathcal B})^*\Delta_{\mathcal C}\cong\Gamma(F).
	\]
	For that, we expand
	\begin{align*}
	(S\circ i_{\mathcal A^-},S\circ i_{\mathcal B})^*\Delta_{\mathcal C} &=\tw\left(
	\begin{tikzcd}[ampersand replacement=\&]
	(i_{\mathcal B}\circ F, i_{\mathcal B})^*\Delta \ar[d,"T"] \& (i_{\mathcal B}\circ F, i_{\mathcal A^-}\circ L[-1])^*\Delta \ar[d,"T"]\ar[l,"U" above]\\
	(i_{\mathcal A^-}[-1], i_{\mathcal B})^*\Delta \& (i_{\mathcal A^-}[-1], i_{\mathcal A^-}\circ L[-1])^*\Delta\ar[l,"U" above]
	\end{tikzcd}
	\right)\\
	&\cong\tw\left(
	\begin{tikzcd}[ampersand replacement=\&]
	\Gamma(F) \ar[d,"\mathrm{id}"] \& \mathbf{0}\\
	\Gamma(F)[-1] \& \Gamma(F)\ar[l,"\mathrm{id}" above]
	\end{tikzcd}
	\right)\\
	&\cong\Gamma(F).
	\end{align*}
	
	\emph{\ref{item:pos spherical swap}$\Rightarrow$\ref{item:spherical swap}:} Tautology.
	
	\emph{\ref{item:spherical swap}$\Rightarrow$\ref{item:F spherical}:} This is immediate from Theorem \ref{thm:HL-S}, because $S$ witnesses the $4$-periodicity of the semiorthogonal decomposition of $\mathcal C$.
	
	We now prove the last statement. To begin, assume $S$ is a $\mathcal B$-positive spherical swap of $F$. By definition, this means that
	\begin{equation}\label{B-positivity using adjoints}
	\mathrm{Id}_{\mathcal B}\cong i_{\mathcal B}^L\circ S\circ i_{\mathcal B},
	\end{equation}
	where $i_{\mathcal B}^L$ is the left adjoint of $i_{\mathcal B}$ (and similar notation will apply for other upcoming adjoints). Because $S\circ i_{\mathcal B}$ factors through $i_{\mathcal B^{\perp\perp}}$, we can expand the above formula to
	\begin{align*}
	\mathrm{Id}_{\mathcal B}&\cong i_{\mathcal B}^L\circ i_{\mathcal B^{\perp\perp}}\circ i_{\mathcal B^{\perp\perp}}^L\circ S\circ i_{\mathcal B}\\
	&=\mathbb L_{\mathcal A^-}\circ i_{\mathcal B^{\perp\perp}}^L\circ S\circ i_{\mathcal B},
	\end{align*}
	meaning that
	\[
	\mathbb R_{\mathcal A^-}\cong i_{\mathcal B^{\perp\perp}}^L\circ S\circ i_{\mathcal B}.
	\]
	Next, take the right adjoint of \eqref{B-positivity using adjoints} to obtain
	\[
	\mathrm{Id}_{\mathcal B}\cong i_{\mathcal B}^R\circ S^{-1}\circ i_{\mathcal B},
	\]
	and argue as above to conclude that
	\[
	\mathbb R_{\mathcal A^+}\cong i_{\mathcal B}^R\circ S\circ i_{\mathcal B^{\perp\perp}}.
	\]
	Now Theorem \ref{thm:HL-S} gives the desired isomorphism
	\begin{align*}
	\mathbf{W}'&\cong\mathbb R_{\mathcal A^+}\circ\mathbb R_{\mathcal A^-}\\
	&\cong i_{\mathcal B}^R\circ S\circ i_{\mathcal B^{\perp\perp}}\circ i_{\mathcal B^{\perp\perp}}^L\circ S\circ i_{\mathcal B}\\
	&=i_{\mathcal B}^R\circ S^2\circ i_{\mathcal B}.
	\end{align*}
	
	For $S$ $\mathcal A$-positive, we will have to take advantage of the fact that $i_{\mathcal A^-}$ and $i_{\mathcal A^+}$ have the same domain $\mathcal A$. The formula $S\circ i_{\mathcal A^-}\cong i_{\mathcal A^+}$ implies
	\[
	i_{\mathcal A^+}^L\circ S\circ i_{\mathcal A^-}\cong\mathrm{Id}_{\mathcal A}.
	\]
	Using the formula $i_{\mathcal A^+}=\cone(T)[-1]$ from the proof of Lemma \ref{lem:functor target is admissible}, we also see that
	\[
	\mathbb R_{\mathcal B}=i_{\mathcal A^-}^R\circ i_{\mathcal A^+}\cong\mathrm{Id}_{\mathcal A}[-1],
	\]
	whence it follows that
	\[
	\mathbb L_{\mathcal B}\cong\mathrm{Id}_{\mathcal A}[1].
	\]
	Now, by Theorem \ref{thm:HL-S},
	\begin{align*}
	\mathbf M'&\cong\mathbb L_{\mathcal B}\circ\mathbb L_{\mathcal B^{\perp\perp}}\\
	&\cong\mathbb L_{\mathcal B^{\perp\perp}}[1]\\
	&=i_{\mathcal A^-}^L\circ i_{\mathcal A^+}[1]\\
	&\cong i_{\mathcal A^+}^L\circ S\circ i_{\mathcal A^+}[1]\\
	&\cong i_{\mathcal A^+}^L\circ i_{\mathcal A^-}\circ i_{\mathcal A^-}^L\circ S\circ i_{\mathcal A^+}[1]\\
	&=\mathbb L_{\mathcal B}\circ i_{\mathcal A^-}^L\circ S\circ i_{\mathcal A^+}[1]\\
	&\cong i_{\mathcal A^-}^L\circ S\circ i_{\mathcal A^+}[2]\\
	&\cong i_{\mathcal A^+}^L\circ S\circ i_{\mathcal A^-}\circ i_{\mathcal A^-}^L\circ S\circ i_{\mathcal A^+}[2]\\
	&=i_{\mathcal A^+}^L\circ S^2\circ i_{\mathcal A^+}[2]
	\end{align*}
\end{proof}

\subsection{Swappable stops}
Our goal is to introduce a geometric analog of Section \ref{sec:spherical swaps}, but we will begin with some gentle discussion of pushoffs of Liouville hypersurfaces. Our starting point will thus be a Liouville pair $P\subset\partial_\infty\bar M$. The assumption that $P$ is a Liouville domain means that we are given a contact form $\alpha$ on $\partial_\infty\bar M$ whose Reeb vector field $R$ is transverse to $P$, and we fix this contact form once and for all.

\begin{lemma}\label{lem:Reeb pushoffs coinitial}
	Let $P'$ be any positive pushoff of $P$, meaning that there is some isotopy of Liouville hypersurfaces $P_t$ from $P_0=P$ to $P_1=P'$ for which $\alpha\left(\frac{\partial}{\partial t}P_t\right)>0$ and $\frac{\partial}{\partial t}P_t$ points to the same side of $P_t$ as $R$. Then there is some $\varepsilon>0$ such that $P'$ is a positive pushoff of $\phi_R^\varepsilon P$.
	
	In other words, the set of Reeb pushoffs of $P$ is coinitial in the set of positive pushoffs of $P$.
\end{lemma}
\begin{proof}
	For $\varepsilon$ sufficiently small, $\phi_R^{-\varepsilon t}P_t$ is still a positive isotopy of Liouville hypersurfaces, because positivity is an open condition. It is contact if and only if $P_t$ itself is. Now compose it with the contactomorphism $\phi_R^\varepsilon$.
\end{proof}

Write $P^-$ and $P^+$ for small negative and positive Reeb pushoffs of $P$, respectively.

\begin{lemma}\label{lem:contact swap characterizations}
	Let $P_i$ be a union of connected components of $P$. The following are equivalent.
	\begin{enumerate}
		\item\label{item:swap pos isotopy} For sufficiently small pushoffs $P_i^\pm$, there is a positive contact isotopy from $P_i^+$ to $P_i^-$ in $\partial_\infty\bar M\setminus P$.
		\item\label{item:swap isotopy} There is a contact isotopy from $P_i^+$ to $P_i^-$ in $\partial_\infty\bar M\setminus P$.
		\item\label{item:swap symplecto} There is a symplectomorphism of $\bar M$ fixing $P\setminus P_i$ and exchanging $P_i^+$ with $P_i^-$.
		\item\label{item:swap contacto} There is a contactomorphism of $\partial_\infty\bar M$ fixing $P\setminus P_i$ and exchanging $P_i^+$ with $P_i^-$.
	\end{enumerate}
\end{lemma}
\begin{proof}
	\emph{\ref{item:swap pos isotopy}$\Rightarrow$\ref{item:swap isotopy}} and \emph{\ref{item:swap symplecto}$\Rightarrow$\ref{item:swap contacto}} are obvious. \emph{\ref{item:swap isotopy}$\Rightarrow$\ref{item:swap symplecto}} follows from Moser's lemma for Liouville manifolds.
	
	\emph{\ref{item:swap contacto}$\Rightarrow$\ref{item:swap pos isotopy}:} Write $\phi$ for the supplied contactomorphism of $\partial_\infty\bar M$, and write $P^t$ for the short positive contact isotopy from $P^0=P_i^-$ to $P^1=P_i^+$ coming from Reeb flow. The image $\phi P^t$ is a positive contact isotopy from $P_i^+$ to $P_i^-$, but it is not guaranteed to stay disjoint from $P_i$ itself. Instead, consider the concatenated path $(\phi P^t)\ast P^t$ is a positive contact isotopy from $P_i^+$ to itself. Because each piece is globally embedded, the concatenated path does not intersect $P_i^+$ except at its endpoints. Composing with a global negative Reeb flow allows us to replace $P_i^+$ by $P_i$, and Lemma \ref{lem:Reeb pushoffs coinitial} lets us turn the loop into the desired isotopy.
\end{proof}
\begin{rmk}
Note that the proof works just as well for $P_i$ a Legendrian submanifold instead of a Liouville hypersurface. In this case, all isotopies can be extended to contact isotopies, so it is equivalent to ask for any Legendrian isotopy from $P_i^+$ to $P_i^-$ in $\partial_\infty\bar M\setminus P$.
\end{rmk}

\begin{defn}\label{def:swappable stop}
	We will say a Legendrian submanifold or Liouville hypersurface $P_i$ is \emph{swappable} if there is \emph{any} isotopy from $P_i^+$ to $P_i^-$ in $\partial_\infty\bar M\setminus P$. In particular, $P_i$ is swappable whenever it satisfies Lemma \ref{lem:contact swap characterizations}.
	
	We will say a stop is swappable if it comes from a swappable Liouville hypersurface.
\end{defn}

Given a swappable stop $\sigma$ of $M$ with fiber $\bar F$, let $\phi$ be an isotopy as in Definition \ref{def:swappable stop}. We obtain autoequivalences $\mathbf M_\phi$ of $\W(\bar F)$ called ``monodromy'', $\mathbf W_\phi$ of $\W(M)$ called ``wrap once'', and $\mathbf S_\phi$ of $\W(M\glu{\sigma}{\sigma_0}\bar F\langle2\rangle)$ called ``swap''. $\mathbf M_\phi$ is the autoequivalence coming from the Liouville automorphism
\[
P^+\overset{\phi}{\cong}P^-\overset{R}{\cong}P^+\qquad\mbox{($\bar F$ is canonically identified with the completion of $P$)},
\]
where the isomorphism $\phi$ comes from applying Moser's lemma to the family $\phi_t(P^+)$, and the isomorphism $R$ comes from Reeb flow. Similarly, $\mathbf W_\phi$ is the autoequivalence realizing deformation invariance for the corresponding family of sectors, which can be built from a zigzag of trivial inclusions.

For $\mathbf S_\phi$, note that $M\glu{\sigma}{\sigma_0}\bar F\langle2\rangle$ is just $M$ with two copies of $\sigma$, which for obvious reasons we will call $\sigma^-$ and $\sigma^+$. The isotopy $\phi$ moving $\sigma^+$ to $\sigma^-$ extends to a deformation of $M\glu{\sigma}{\sigma_0}\bar F\langle2\rangle$ which moves $\sigma^-$ to $\sigma^+$ by the minimal counterclockwise rotation (cf. Proposition \ref{prop:A_n sector presentation}). Declare $\mathbf S_\phi$ to be the autoequivalence induced by this deformation. It is immediate from Lemma \ref{lem:Orlov deformation invariance} that $\mathbf S_\phi$ exchanges the images of $\imath_{\sigma^-}$ and $\imath_{\sigma^+}$.

\begin{proof}[Proof of Theorem \ref{thm:swappable stops give spherical Orlov functors}]
	We wish to apply Proposition \ref{prop:spherical swap iff spherical}, which means we need a geometric interpretation of the category $\mathcal C$ associated to the functor $\imath_\sigma$. We claim that
	\[
	\Perf\,\mathcal C\cong\Perf\,\W(M\dglu{\sigma}{\sigma_0}\Sigma\bar F)\cong\Perf\,\W(\Sigma\bar F\dglu{\sigma_0}{\sigma}M).
	\]
	Indeed, by Corollary \ref{cor:directed gluing formula}, the first directed gluing gives a semiorthogonal presentation $\langle\W(M),\W(\bar F)\rangle$, with $\W(\bar F)$ mapping in by $(i_{\Sigma F})_*\circ\imath_1[1]$. The gluing bimodule is
	\[
	\Gamma^\dagger(\imath_\sigma)\otimes_{\W(\bar F)}\Gamma((\imath_{\sigma_1}[1])^{-1}\imath_{\sigma_0})[1]\cong\Gamma^\dagger(\imath_\sigma)[1],
	\]
	so we reparametrize the $\W(\bar F)$ factor to remove both shifts. By the same argument, $\W(\Sigma\bar F\dglu{\sigma_0}{\sigma}M)$ has a semiorthogonal decomposition $\langle\W(\bar F),\W(M)\rangle$, where $\W(\bar F)$ is parametrized by $(i_{\Sigma F})_*\circ\imath_1[2]$, and the gluing bimodule is $\Gamma(\imath_\sigma)$.
	
	Both of these semiorthogonal decompositions come from gluing descriptions of $M\glu{\sigma}{\sigma_0}\bar F\langle2\rangle$, and by treating the inclusion of $M$ as fixed we see that the two copies of $\W(\bar F)$ come from the two stops. It follows that $\mathbf S_\phi$ is a spherical swap of $\imath_\sigma$, that $\mathbf W_\phi$ is the restriction of $S^2$ to $\W(M)$, and that $\mathbf M_\phi$ is the restriction of $S^2$ to the second copy of $\W(\bar F)$.
	
	It remains to check that $\mathbf S_\phi$ is positive, so let us compare the copies of $\W(\bar F)$. The first gluing description $M\dglu{\sigma}{\sigma_0}\Sigma\bar F$ has $\W(\bar F)$ parametrized by
	\[
	(i_{\Sigma F})_*\circ\imath_1\cong(i_{\Sigma F})_*\circ\imath_0[-1]\cong(i_{\bar F\langle2\rangle})_*\circ\imath_0[-2].
	\]
	This corresponds to a minimal clockwise rotation (through $\sigma_1(\bar F\langle2\rangle)=\sigma^+$) of
	\[
	(i_{\bar F\langle2\rangle})_*\circ\imath_2[-2]\cong(i_M)_*\circ\imath_\sigma[-1].
	\]
	Noting that the rotation takes place in $\bar F\times S^1\subset\partial_\infty\bar F\langle2\rangle$ so that a full rotation is a shift by $2$, this is the same as a minimal counterclockwise rotation of $(i_M)_*\circ\imath_\sigma[1]$.
	
	The second gluing description $\Sigma\bar F\dglu{\sigma_0}{\sigma}M$ has $\W(\bar F)$ parametrized by
	\[
	(i_{\Sigma F})_*\circ\imath_1[2]\cong(i_{\Sigma F})_*\circ\imath_0[1]\cong(i_{\bar F\langle2\rangle})_*\circ\imath_2,
	\]
	though of course this is with respect to a different numbering and grading of the stops of $\bar F\langle2\rangle$. This corresponds to a minimal counterclockwise rotation of
	\[
	(i_{\bar F\langle2\rangle})_*\circ\imath_0\cong(i_M)_*\circ\imath_\sigma[1]
	\]
	through $\sigma^-$. Since $\mathbf S_\phi$ was defined on $\sigma^-$ by minimal counterclockwise rotation, it follows by Lemma \ref{lem:Orlov deformation invariance} that $\mathbf S_\phi$ is $\mathcal A$-positive.
	
	For $\mathcal B$-positivity, we need to check that $\mathbf S_\phi\circ(i_M)_*$ becomes isomorphic to $(i_M)_*$ after removing the stop which starts as $\sigma^+$, which follows from Proposition \ref{prop:isotopic inclusions}.
\end{proof}

	\bibliographystyle{plain}
	\bibliography{large-bib}

\end{document}